\documentclass[12pt]{amsart}

\usepackage[utf8]{inputenc}
\usepackage[T1]{fontenc}
\usepackage{cmap}

\usepackage{amsfonts, mathtools,latexsym, array, amssymb, mathrsfs, stmaryrd}
\usepackage{bbm}
\usepackage{psfrag}
\usepackage[all]{xy}
\usepackage{graphicx}
\usepackage{enumerate}
\usepackage{enumitem}
\usepackage[usenames,dvipsnames]{color}
\usepackage{circledsteps}

\newcommand{\annotation}[1]{\marginpar{\tiny #1}}

\newcommand{\comment}[1]{}

\newcommand\JB{\color{black}}
\newcommand{\jb}[1]{\annotation{\JB{#1}}}  
\newcommand\BK{\color{black}}

\theoremstyle{plain} 
\newtheorem{maintheorem}{Theorem}

\numberwithin{equation}{section}
\newtheorem{theorem}[equation]{Theorem}

\newtheorem{addendum}[equation]{Addendum}

\newtheorem{corollary}[equation]{Corollary}
\newtheorem{lemma}[equation]{Lemma}
\newtheorem{fact}[equation]{Fact}

\newtheorem{proposition}[equation]{Proposition}
\newtheorem{claim}[equation]{Claim}

\newtheorem*{claim*}{Claim}

\newtheorem{definition}[equation]{Definition}

\newtheorem{remark}[equation]{Remark}
\newtheorem{remarks}[equation]{Remarks}

\newcommand\cC{\mathcal C}

\newcommand\cM{\mathcal M}

\newcommand\erg{{\textnormal{erg}}}
\newcommand\eps{\varepsilon}

\newcommand\NLP{\Pi}
\newcommand\FNLP{\boldsymbol{\Pi}}
\newcommand\NN{{\mathbb N}}

\newcommand\RR{{\mathbb R}}

\newcommand\ZZ{{\mathbb Z}}

\newcommand\Prob{{\operatorname{\mathscr{P}}}}
\newcommand\Proberg{\Prob_{\kern-3pt\mathrm{erg}}}

\newcommand\temp{\beta}

\renewcommand\top{{\mathrm{top}}}

\def\eg{{\it e.g.\ }}
\newcommand\vpot{\vec\pot}
\newcommand\dom{\operatorname{dom}}
\newcommand\inter{\operatorname{int}}

\DeclareMathOperator{\interior}{int}

\newcommand{\En}{\mathcal{E}}
\newcommand{\wass}{\operatorname{W}}
\newcommand{\diam}{\operatorname{diam}}
\newcommand*{\dd}%
  {\relax\ifnum\lastnodetype>0\mskip\medmuskip\fi\mathrm{d}}

\newcommand{\rs}{\operatorname{\rho}}
\newcommand{\grad}{\operatorname{\nabla}}
\newcommand{\hf}{\mathsf{h}}
\newcommand{\Pf}{\mathsf{P}}
\newcommand{\EM}{\mathscr{E\kern-2pt M}}

\newcommand{\pot}{\varphi}

\newcounter{In}
\newcommand{\In}{\refstepcounter{In}\Circled{\theIn}}
\newcommand{\Inref}[1]{\Circled{\ref{#1}}}

\newcommand\ffi{\varphi}
\newcommand\fios{\varphi_{OS}}

\def\s{\sigma}
\def\S{\Sigma}

\def\al{\alpha}
\def\be{\beta}

\def\8{\infty}
\def\disp{\displaystyle}

\newcommand{\BBone}{\mathbbm{1}}
\newcommand{\ninf}{{n\to+\8}}

\def\N{\mathbb{N}}
\def\R{\mathbb{R}}
\def\P{\mathbb{P}}

\newcommand{\Z}{{\mathbb Z}}

\newcommand{\CP}{\mathcal{P}}

\newcommand{\ol}{\overline}
\newcommand{\ul}{\underline}

\newcommand\hS{\widehat{S}}

\begin{document}

\title{Nonlinear thermodynamical formalism}


\author{J\'er\^ome Buzzi}
\address{Laboratoire de Math\'ematiques d'Orsay - CNRS \& Universit\'e Paris-Sud}
\thanks{JB was partially supported by the ISDEEC project ANR-16-CE40-0013.}
\email{jerome.buzzi@universite-paris-saclay.fr}

\author{Beno\^{\i}t R. Kloeckner}
\address{Univ Paris Est Creteil, CNRS, LAMA, F-94010 Creteil, France;
Univ Gustave Eiffel, LAMA, F-77447 Marne-la-Vallée, France}
\email{benoit.kloeckner@u-pec.fr}

\author{Renaud Leplaideur}
\thanks{RL wants to thank ERC project 692925 NUHGD for kind support for a visit to Orsay in September 2018}
\address{ISEA, Universit\'e de la Nouvelle-Cal\'edonie \& LMBA UMR6205}
\email{renaud.leplaideur@unc.nc}

\begin{date}
{\today}
\end{date}

\begin{abstract}
We define a nonlinear thermodynamical formalism which translates into dynamical system theory the statistical mechanics of  generalized mean-field models, extending investigation of the quadratic case by Leplaideur and Watbled.

Under suitable conditions, we prove a variational principle for the nonlinear pressure and we characterize the nonlinear equilibrium measures and relate them to specific classical equilibrium measures.

In this non-linear thermodynamical formalism, which can, e.g., model mean-field approximation of large systems, several kind of phase transitions appear, some of which cannot happen in the linear case. We use our correspondence between non-linear and linear equilibrium measures to further the understanding of phase transitions, { both in previously known cases (Curie-Weiss and Potts models) and in} new examples (metastable phase transition).

Finally, we apply  some of the ideas introduced to the classical thermodynamical formalism, proving that freezing phase transitions can occur over \emph{any} zero-entropy invariant compact subset of the phase space.
\end{abstract} 
\maketitle


\section{Introduction}

In the 1970s, Sinai, Ruelle, Bowen, and others (see, e.g., \cite{Sinai-GibbsMeasures,Ruelle-book,Bowen-book}) developed a thermodynamical approach to dynamical systems inspired by the statistical mechanics of lattice systems. 
In a recent work \cite{Leplaideur-Watbled}, the third named author and Watbled  applied this program to the Curie-Weiss mean-field theory: they introduced a new thermodynamical formalism over the full shift where the energy functional is quadratic. They obtained precise results using the specific structure of this setting. 

Our goal in this paper is to understand the generality of their results. It turns out that we can define the nonlinear pressure of a measure as the sum of its entropy and its ``energy'', defined as \emph{any weak-star continuous function} of the measure. We are in particular interested in the case when the energy is a smooth function of the integrals of one or several potentials, in which case we call it an \emph{energy with potential(s)}. Assuming only that the classical thermodynamical formalism is well-behaved, we can analyze this \emph{nonlinear} thermodynamics using suitable convex analysis.

We first prove a variational principle: the supremum of the nonlinear pressure of the measures is given by a combinatorial formula involving the classical  separated sets for the Bowen-Dinaburg dynamical metric (Theorem \ref{thm-var-prin-NL}), then defining \emph{equilibrium measure} as those measures achieving the previous supremum. It is easy to show that equilibrium measures exist and, in the expansive case, we relate them to Gibbs ensembles (Theorem \ref{thm-existence-NL}). In the case of energies with potentials we  show that equilibrium measures are classical equilibrium measures for some specific linear combination of these potentials (Theorem \ref{thm-equilibria}).
When the nonlinearity is a real-anaytic function of the integral of a single potential, we obtain finiteness of the set of equilibrium measures (Theorem \ref{thm-finiteness}). As is well-known from physics and examples including the Curie-Weiss theory, \emph{phase transitions} can occur in this setting, e.g., there may be several equilibrium measures that may depend non-analytically on parameters giving rise to \emph{freezing} (Theorem \ref{thm-fpt} and Section \ref{sec:transition}) or \emph{metastable phase transitions.} (Theorem \ref{th-metastable} in Section \ref{sec-metastable}).

\subsection{Classical thermodynamical formalism}
We recall the classical definitions (see, e.g., \cite{Walters-book}). We will sometimes call these notions \emph{linear} to distinguish them from the ones we introduce in this paper.

Let $(T,\pot)$ be a continuous system, i.e., a continuous self-map $T:X\to X$ of a compact metric space  together with a continuous function $\pot \in C(X,\RR)$. 
The function $\pot$ is called the \emph{potential}.  We denote by $\Prob$ the set of Borel probability measures on $X$, endowed with the weak star topology, by $\Prob(T)$ the subset of $T$-invariant measures and by $\Proberg(T)$ the subset of ergodic and invariant measures.

The weight of order $n$ of a finite subset $C\subset M$ is:
 $$
    w_n(C) := \sum_{x\in\mathcal C} \exp \left(S_n\pot(x)\right)
 $$ 
 where $S_n\pot$ denotes a Birkhoff sum:
 $$
   S_n\pot(x):=\pot(x)+\pot(Tx)+\dots+\pot(T^{n-1}x).
 $$
Given $\eps>0$ and $n\in\NN$, the \emph{Bowen-Dinaburg dynamical balls} are the sets
 $$
    B(x,\eps,n):=\big\{y\in X: \forall 0\leq k<n,\; d(T^ky,T^kx)<\eps \big\}.
 $$
A finite set $\mathcal C$ is an \emph{$(\eps,n)$-covering} when $\bigcup_{x\in \mathcal C} B(x,\eps,n)=X$. It is an \emph{$(\eps,n)$-separated} subset when for all distinct $x, x'\in \mathcal{C}$, $x'\notin B(x,\eps,n)$. The \emph{partition function} is:
 $$
    Z(\eps,n) := \sup_{\mathcal C} w_n(\mathcal C)
 $$
where $\mathcal C$ ranges over the $(\eps,n)$-separated subsets of $X$.
  
An $(\eps,n)$-separated set $\mathcal C$ is said to be \emph{adapted} when it realizes the supremum in $Z(\eps,n)$, and each adapted set defines a probability measure
\begin{equation}
    \mu_{\mathcal{C}} := \frac{1}{Z(n,\eps)} \sum_{x\in\mathcal C} e^{S_n\phi(x)} \frac{\delta_x+\delta_{Tx}+\dots+\delta_{T^{n-1}x}}{n}.
 \end{equation}
called an \emph{$(\eps,n)$-Gibbs ensemble}.

The (linear) topological pressure is:
  \begin{equation}\label{eq-VP-classic}
  P_\top(T,\pot):=
    \lim_{\eps\to0} \limsup_{n\to\infty}\frac1n\log Z(\eps,n).
  \end{equation}

  Recall that the (linear) pressure of a measure $\mu\in\Prob(T)$ with respect to the potential $\pot$ is (denoting by $h$ the Kolmogorov-Sinai entopy):
 $$
     P(T,\pot,\mu):= h(T,\mu) +   \int\pot\,d\mu.
  $$
  
The variational principle states  that:
 \begin{equation}\label{eq-VP}
        P_\top(T,\pot)  =  \sup_{\mu\in\Prob(T)} P(T,\pot,\mu).
 \end{equation}
 An \emph{equilibrium measure} for $(T,\pot)$ is then an invariant probability measure $\mu$ such that
 $P(T,\pot,\mu)=P_\top(T,\pot)$, i.e., a measure that achieves the above supremum.

\medskip
{ The (linear) pressure function is the function $\be\mapsto P_{\top}(T,\beta\pot)$ where $\be$ is a real parameter, called the \emph{inverse of temperature}.}


\subsection{Nonlinear formalism}
We propose the following generalization. It will prove convenient to write $\mu(\pot)$ for $\int\pot \dd\mu$. We consider again a continuous map $T$ acting on a compact metric space $X$.

An \emph{energy} is a function $\En : \Prob \to \mathbb{R}$ which is  continuous in the weak star topology; note that we will need the energy of non-invariant measures. We say that $\En$ is an energy \emph{with potential} $\pot$ (a continuous function defined on $X$) if it can be written
 $$
    \En(\mu) = F\big(\mu(\pot)\big)
 $$
for some continuous function $F:I\to\mathbb{R}$ defined on an interval containing all values taken by $\pot$. More generally, an \emph{energy with potentials} takes the form
 \begin{equation}\label{eq-Energy}
     \En(\mu) = F\big(\mu(\pot_1), \dots, \mu(\pot_d)\big)
 \end{equation}
where $\pot_1,\dots,\pot_d$ are continuous functions defined on $X$ and $F:U\to\mathbb{R}$ is a continuous function on some set $U\subset\RR^d$. For $\En$ to be well-defined on the whole of $\Prob$, the set $U$ must contain the convex hull of the set of values taken by $\vec\pot = (\pot_1,\dots,\pot_d):X\to\mathbb{R}^d$.
We add the adjective ``$C^r$'' ($r\ge 1$), ``smooth'' or ``analytic'' to $\En$ whenever the domain $U$ of $F$ is open and $F$ is $C^r$ ($r=\infty$ meaning smooth, $r=\omega$ meaning analytic) on $U$.

An energy is said to be \emph{convex} when for all Borel probability measure $\xi$ on $\Prob$ (hence, $\xi$ is a measure of measures):
 $$
    \En\Big(\int \nu \dd\xi(\nu) \Big) \le \int \En(\nu) \dd\xi.
 $$
For example, if $\En$ is an energy with potentials, it is convex whenever $F$ is.
\bigskip

Not assuming potentials, we first need to replace Birkhoff sums. Given $x\in X$ and $n\in\mathbb{N}$, we define an \emph{empirical measure} 
 $$
    \Delta_x^n := \frac1n \sum_{i=0}^{n-1} \delta_{T^ix} .
 $$
Observe that for any potential $\pot$, $\Delta_x^n(\pot)=\frac1n S_n\pot(x)$ is the averaged Birkhoff sum. We thus define the \emph{nonlinear weight} of order $n$ of a finite set $\mathcal C$ and the \emph{nonlinear partition function} as 
 $$
    \omega_n(\mathcal{C}) = \sum_{x\in \mathcal{C}} e^{n\En(\Delta_x^n)} \qquad
    \zeta(\eps,n) = \sup_{\mathcal{C}} \omega_n(\mathcal{C}),
 $$
where the supremum is taken over all $(\eps,n)$-separated sets $\mathcal C$.

Again, an $(\eps,n)$-separated set $\mathcal C$ is said to be \emph{adapted} if it realizes the maximum in $\zeta_n$ and we define an \emph{nonlinear $(\eps,n)$-Gibbs ensemble} 
\begin{equation}\label{eq-def-ensemble}
    \mu_{\mathcal C} := \frac{1}{\zeta(\eps,n)} \sum_{x\in\mathcal C} e^{n \En(\Delta_x^n)} \Delta_x^n \quad \in\Prob
 \end{equation}
(note that the continuity of $T$ ensures that the maximum in $\zeta$ is realized for all $(\eps,n)$).

The \emph{nonlinear topological pressure}, to be thought of as an analogue of topological entropy weighted by energy, is 
\begin{equation}
    \NLP_\top^\En(T) = \lim_{\eps\to 0} \limsup_{n\to\infty} \frac1n \log \zeta(\eps,n).
\label{eq:NLP}
\end{equation}
In Theorem \ref{thm-var-prin-NL} we will show that under suitable hypotheses, replacing the supremum limit by an infimum limit:
 $$
    \underline{\NLP}_\top^\En(T) =  \lim_{\eps\to 0} \liminf_{n\to\infty} \frac1n \log \zeta(\eps,n)
 $$
gives the same quantity as $\NLP_\top^\En(T) $.
Meanwhile the \emph{nonlinear pressure} is defined for all invariant probability measures $\mu$ by
 $$
    \NLP^\En(T,\mu) = h(T,\mu) + \mathcal{E}(\mu).
 $$

\subsection{Main results}

For certain nonlinear systems $(T,\En)$, it may happen that some measures satisfy $\NLP^\En(T,\mu) > \NLP_\top^\En(T)$; we will first give conditions excluding this.

\begin{definition} 
We will say that $(T,\En)$ has an \emph{abundance of ergodic measures} if for any $\mu\in\Prob(T)$ and $\eps>0$, there is an ergodic measure $\nu\in \Proberg(T)$ such that { $h(T,\nu)+\En(\nu)> h(T,\mu)+\En(\mu)-\eps$}.
\end{definition}
 This condition is satisfied by uniformly hyperbolic diffeomorphisms that have a single basic set in their spectral decomposition as any invariant probability measure can be approximated by an ergodic one { both in the weak star topology and in entropy}. It is also satisfied for arbitrary continuous systems $(T,\En)$ with convex energy, since, in this case, for any $\mu\in\Prob(T)$,
 $$
   \NLP^\En(T,\mu) \le \int h(T,\mu_\xi) + \mathcal{E}(\mu_\xi)\, dP(\xi)
 $$
using the ergodic decomposition $\mu=\int \mu_\xi\, dP(\xi)$.


\medbreak

Recall that, in the invertible case, $T$ is said to be an \emph{expansive homeomorphism} when there exist a number $\eps_0>0$ (called an \emph{expansivity constant} for $T$) such that
 $$
  \forall x,y\in X\quad \sup_{n\in\ZZ} d(T^nx,T^ny)\leq\eps_0 \implies x=y
 $$
(see, e.g., \cite{Katok-Hasselblatt} Definition 3.2.11; note that here we use a $\le$ sign, making the expansivity constants possibly slightly smaller).
This notion is generalized to non-necessarily invertible maps under the name of \emph{positive expansivity} by considering only the positive orbits:
 $$
  \forall x,y\in X\quad \sup_{n\ge0} d(T^nx,T^ny)\leq\eps_0 \implies x=y.
 $$
and the results we state below for expansive homeomorphisms could be extended to positively expansive map with the same proofs.
\bigbreak
 
Our first result establishes a variational principle generalizing eq.~\eqref{eq-VP} to all energies.

\begin{maintheorem}[Variational principle]\label{thm-var-prin-NL}
Let $T:X\to X$ be a continuous map of a compact space and let $\En:\Prob\to\RR$ be an energy. 
{\JB Assume that $(T,\En)$  has an abundance of ergodic measures,} 

Then the nonlinear topological pressure satisfies:
  \begin{equation}\label{eq-def-Ptop}
      \sup_{\mu\in\Prob(T)} \NLP^\En(T,\mu) = \NLP^\En_\top(T) = \underline{\NLP}_\top^\En(T)
  \end{equation}
  {\JB If, additionally,} $T$ is an expansive homeomorphism with some constant $\eps_0>0$, then
  $$
   \NLP^\En_\top(T) =\lim_n \frac1n \log \zeta(\eps_0,n).
   $$
\end{maintheorem}

When the conclusion $\sup_{\Prob(T)} \NLP^\En(T,\cdot) = \NLP^\En_\top(T)$ of the {\BK above} theorem holds, we define a \emph{nonlinear equilibrium measure} as any measure $m\in\Prob(T)$ realizing this supremum:
  $$
    \NLP^\En(T,m)=\max_{\Prob(T)} \NLP^\En(T,\cdot).
  $$

As in the classical setting, existence of an equilibrium measure is easily  obtained  when entropy is upper semicontinuous, and in the expansive case equilibrium measures prescribe the asymptotic behavior of Gibbs ensembles.

\begin{maintheorem}\label{thm-existence-NL}
{\JB Let $T:X\to X$ be a continuous map of a compact space and let $\En:\Prob\to\RR$ be an energy.  Assume that $(T,\En)$  has an abundance of ergodic measures.}  

If $\mu\mapsto h(T,\mu)$ is upper semicontinuous\footnote{This holds, e.g., if $(X,T)$ is a subshift \cite{Walters-book} or a $C^\infty$ smooth map \cite{Buzzi-SIM}.}, then there exists at least one nonlinear equilibrium measure.
 
If additionally $T$ is an expansive homeomorphism for some constant $\eps_0>0$, then any accumulation point $\mu$ of any sequence $(\mu_{\mathcal{C}_n})_{n\in\mathbb{N}}$ of nonlinear Gibbs $(\eps_0,n)$-ensembles belongs to the closure of the convex span of all nonlinear equilibrium measures.
\end{maintheorem}

The last statement means that there exists a probability measure $\xi$ on $\Prob$ (a measure of measures), concentrated on the set $\EM$ of equilibrium measures, such that
\[ \mu = \int_{\EM} \nu \dd\xi(\nu)\]
(see, e.g., \cite{Phelps-book}, Proposition 1.2.) The accumulation points can indeed fail to be equilibrium measures, e.g., in the Curie-Weiss model when there are two asymmetric equilibrium measures and one chooses symmetric Gibbs ensembles, see \cite{Leplaideur-Watbled}.

\bigbreak
Next we study the uniqueness and nature of the nonlinear equilibrium measures in the case of an energy with potentials as in eq.~\eqref{eq-Energy}. Our main point here is that we can use classical convex analysis to reduce the nonlinear thermodynamical formalism to the linear one.
 
More precisely, we will use the classical Legendre duality between entropy and pressure; using the vector of integral of potentials $(\mu(\pot_1),\dots,\mu(\pot_d))$ as intermediate coordinates, this will reduce to finite-dimensional Legendre duality. This duality holds for the class of  \emph{$C^r$ Legendre} systems $(T,\vec\pot)$ (where $r\in\mathbb{N}^*\cup\{\infty,\omega\}$  and $C^\omega$ means analytic), see Definitions \ref{defi-Legendre}, \ref{defi-regular}.
When additionally each linear combination of the $(\pot_i)$ admits a unique linear equilibrium measure, one says that $(T,\vec\pot)$ is \emph{$C^r$ Legendre with unique linear equilibrium measures}.
 
Let us note that classical examples fulfill these requirements: if $T$ is a topologically transitive Anosov diffeomorphism or expanding map, and $(\pot_1,\dots,\pot_d)$ is a family of H\"older-continuous potentials whose linear combinations are not cohomologuous to a constant, i.e., for all $\alpha_1,\dots,\alpha_d\in\mathbb{R}^d$:
 $$
  \Big(\exists u\in C(X,\mathbb{R}), \exists c\in\mathbb{R}\colon  \sum_{i=1}^d \alpha_i\pot_i = u-u\circ T + c \Big)\implies \alpha_1=\dots=\alpha_d=0,
 $$
then $(T,\vec\pot)$ is $C^r$ Legendre with unique linear equilibrium measures.


\begin{maintheorem}\label{thm-equilibria}
Assume that $(T,\vec\pot)$ is {\JB $C^r$ Legendre,} that $F:U\subset\mathbb{R}^d\to\mathbb{R}$ is $C^r$ and consider the energy with potentials $\En(\mu) = F\big(\mu(\pot_1),\dots,\mu(\pot_d)\big)$. Then there is a nonempty compact subset $\mathscr{Y}\subset\RR^d$ such that the nonlinear equilibrium measures are exactly the linear equilibrium measures with respect to each of the potentials $\sum_i y_i\pot_i$ where $(y_1,\dots,y_d)\in\mathscr{Y}$.
\end{maintheorem}

Observe that as a consequence, even though nonlinear equilibrium measures may fail to be unique, under the hypotheses of Theorem \ref{thm-equilibria} they are ergodic as soon as linear equilibrium measures are (and more, see Corollary \ref{cor-individual}).

\begin{addendum}\label{add-equilibria}
In the  above setting, the set $\mathscr{Y}$ can be computed from the linear pressure function defined over $\RR^d$ by
  $
      \Pf(y_1,\dots,y_d) = P_\top(T,\sum_i y_i\pot_i).
  $
More precisely $ \mathscr{Y}=(\nabla \Pf)^{-1}(\mathscr{V})$ where $\nabla \Pf$ is the gradient of $\Pf$ and
\begin{align*}
  \mathscr{V} &=\{z\in\RR^d \colon \hf(z)+F(z)=\sup(\hf + F) \} \\
  \hf(z) &=\sup\{h(T,\mu):\mu(\vec\pot)=z\}.
\end{align*}
The function $\hf$ can also be computed from $\Pf$, as $-\hf$ is the Legendre dual of $P$.
\end{addendum}

\begin{remarks}
Given $(T,\vpot)$ a smooth Legendre system, any compact subset of $\RR^d$ can be realized as the set $\mathscr{Y}$ above by choosing a suitable  $C^\infty$ smooth nonlinearity $F$ {\JB (Corollary~\ref{cor-flex-Y})}. 

Our proof will apply to a more general notion of equilibrium measures, see eq.~\eqref{eq-G}.
\end{remarks}
\bigbreak

\begin{maintheorem}\label{thm-finiteness}
If  $(T,\pot)$ is a $C^\omega$ Legendre system with unique linear equilibrium measures and $F$ is $C^\omega$ with a single potential $(d=1)$, then there are only finitely many nonlinear equilibrium measures.
\end{maintheorem}

Note that we do not simply claim that $\EM$ is finite-dimensional, but that it is finite, even though it can contain several equilibrium measures. In fact, this failure of uniqueness can occur even for a topologically transitive subshift of finite type with a H\"older-continuous potential (see \eg \cite{Leplaideur-Watbled} and Section \ref{sec-examples} below).
However uniqueness holds for generic non-linearities for any $d\ge1$  (Proposition \ref{p-generic-unique}).

\bigbreak
The above characterization shows that for many systems with expanding or hyperbolic properties, such as mixing subshifts of finite type, the nonlinear equilibrium measures share the good ergodic properties of the classical equilibrium measures. Let us recall some of them.

\begin{corollary}[Folklore]\label{cor-individual}
Let $(X,T)$ be a mixing subshift of finite type (not reduced to a fixed point).
Consider H\"older-continuous potentials $\vec\pot:X\to\mathbb{R}^d$ and a $C^r$ nonlinearity $F:U\subset\mathbb{R}^d\to\mathbb{R}$. Then, for the energy given by $\En(\mu) = F\big(\mu(\pot_1),\dots,\mu(\pot_d)\big)$, any nonlinear equilibrium measure 
 \begin{itemize}
   \item is ergodic and mixing;
   \item has exponential decay of correlation;
   \item satisfies the almost sure invariance principle and in particular the central limit theorem.
 \end{itemize}
where the two last properties are understood to hold with respect to H\"older-con\-ti\-nuous observables.
\end{corollary}

These results are folklore in the sense that some of them are immediate consequences of the founding results of Sinai, Ruelle, and Bowen, while others were first considered in more general settings. The following are convenient references:  ergodicity, mixing, and exponential decay of correlation follow from Ruelle's Perron-Frobenius theorem (see, e.g.,\cite[chapter 1]{Baladi-book}), %
the almost sure invariance principle, which implies many limit theorems was proved in \cite{MN05} in much greater generality.

\subsection{Examples}\label{subsec-examples}

We will give a few examples to which the above theorems apply,
mostly inspired by physics. These examples involves an additional real parameter, the inverse temperature $\beta>0$: the energy function is then $\En(\mu)=\beta\En_1(\mu)$ where $\En_1$ is a reference energy and $\beta$ tunes the balance between entropy and that energy, in agreement with thermodynamics.\footnote{In thermodynamics, the equilibrium state of a system in contact with a thermostat at inverse temperature $\beta$ is such that it maximizes  the entropy of the total system (combining the initial system and the thermostat), i.e., the quantity $h(T,\mu)-\beta\En(\mu)$, up to the addition of a constant. As is customary in dynamics, the minus sign has been included in the definition of the energy function.} This leads to the natural question of how the existence, the number, or the equilibrium measures themselves depend on this parameter $\beta$, leading to the physical notion of phase transitions.

\subsubsection{Linear case}

The classical, linear formalism is the special case where $d=1$ and $F(z)= \beta z$ for $\beta>0$ 
and taking any $\pot\in C(X,\RR)$. The nonlinear pressure then coincides with the linear one: $h(T,\mu)+\beta\int \pot \dd\mu$, yielding a first example. Here $\mathscr{Y}=\{\beta\}$.

\subsubsection{Classical Curie-Weiss model}

Consider $X=\{-1,1\}^{\mathbb{N}}$ of $X=\{-1,1\}^{\mathbb{Z}}$, let $T$ be the shift map, and set $\pot(x_0x_1\cdots)=x_0$ and $F(z)=\frac\beta2z^2$ for some $\beta\ge 0$; i.e., maximize $h(T,\mu)+\frac\beta2\mu(\pot)^2$.  The set $\mathscr{Y}$ can have one or two elements depending on the value of $\beta$: see \cite{Leplaideur-Watbled} and Section \ref{sec-Curie-Weiss}. { The notations were slightly different:
 $\NLP^\En_\top(T)$ here  was $\CP_{2}(\pot)$ in \cite{Leplaideur-Watbled}, $\hf(z)$ was $\ol{H}(z)$ and $\hf(z)+F(z)$ was $\ol\varphi(z)$.}

\subsubsection{Asymmetric Curie-Weiss model}

In Section \ref{sec-metastable} we shall give an asymmetric Curie-Weiss model, where $T$ is again a full shift map, $\pot$ is a Bernoulli potential and $F(z)=\frac\beta2z^2$, but exhibiting a \emph{metastable} phase transition: at each temperature there are finitely many \emph{local} maximizers, but at some critical temperature the global maximizer jumps from one local maximizer to another.

\subsubsection{Curie-Weiss-Potts}

{\JB The consideration of several potentials is motivated by} 
the Curie-Weiss-Potts model: $X=\{\theta_1,\dots, \theta_n\}^{\mathbb{N}}$ or $X=\{\theta_1,\dots, \theta_n\}^{\mathbb{Z}}$, $T$ the shift map, $\pot_i(x_0x_1\cdots)=\BBone_{\theta_i}(x_0)$ and 
$F(z)=\frac{\beta}{2} \lVert z\rVert^2$ where $\lVert\cdot\rVert$ is the usual Euclidean norm, exhibiting yet another form of phase transition as $\beta$ varies, see \cite{Leplaideur-Watbled-2} and Section \ref{subsec:potts}.

\subsubsection{Wassertein distance to the maximal entropy measure}\label{sec-Wasserstein}

We can go beyond the case with potentials: let us give a simple but intriguing example. Consider the map $T:x\mapsto 2 x \mod 1 $ on the circle, with reference energy $\En_{1}(\mu)=\wass_p(\mu,\lambda)$ where $\lambda$ denotes the Lebesgue measure, $p\in[1,+\infty)$ and $\wass_p$ is the \emph{Wasserstein distance} of exponent $p$.

Theorems \ref{thm-var-prin-NL} and \ref{thm-existence-NL} ensure that the nonlinear topological pressure is achieved by at least one invariant measure. The main question, which we leave open, is then to describe the non-empty set of equilibrium measures for $\beta\En_1$, in particular determine uniqueness. 

For $\beta=0$, $h(T,\mu)+\beta W_p(\mu,\lambda)$ reduces to the entropy so $\lambda$ is the unique equilibrium. When $\beta\to\infty$, the set of equilibrium measures must converge to $\{\delta_0\}$, since $\delta_0$ is the unique invariant measure maximizing $W_p(\mu,\lambda)$. 

\subsection{More Phase Transitions}\label{subsec-comments}
A phase transition can be defined from any of a number of different phenomena that often {\JB occur simultaneously}: loss of the analyticity of the pressure with respect to physical parameters, multiple equilibrium measures, or failure of the central limit theorem for example. 

Sarig \cite{Sarig-Phase-Transitions} has studied such equivalences
in the setting of Markov shifts. In contrast, we see here  (Section~\ref{sec-Curie-Weiss}) that non-analyticity of pressure and multiplicity of equilibrium measures can occur though the central limit theorem continues to hold (Corollary~\ref{cor-individual}). Such distinctions have been observed before in \cite{leplaideur-butterfly} and \cite{thaler}. 
The key point of view in the definition of Legendre regular systems and the proof of Theorems \ref{thm-equilibria} and \ref{thm-finiteness} is to consider a certain convex set, the \emph{entropy-potential diagram} (defined in Section \ref{sec:convexity}, see figures \ref{fig:diagramme2D}, \ref{fig:diagramme1Da}), which describes the pairs $(h(T,\mu);\mu(\vec\pot))$ that can be achieved when $\mu$ runs over $\Prob(T)$. Phase transitions then occur when the nonlinearity ``becomes more convex'' than the diagram.

In Section \ref{sec:transition}, we shall illustrate more broadly the benefits of this diagram by considering \emph{freezing phase transitions}, by which we mean that for  all $\beta>\beta_0$ for some $\beta_0>0$, the set of equilibrium measures is non-empty and independent of $\beta$; its elements are called ``ground states'' as they must maximize the energy.
\begin{maintheorem}\label{thm-fpt}
Let $T:X\to X$ be a continuous dynamical system of finite, positive topological entropy, and assume that $\mu\mapsto h(T,\mu)$ is upper semi-continuous.
\begin{enumerate}
\item For every $\mu_0\in\Prob_\erg(T)$ with zero entropy there exist a continuous potential $\pot:X\to \mathbb{R}$ such that the \emph{linear} thermodynamical formalism of $(T,\pot)$ exhibits a freezing phase transition with unique ground state $\mu_0$.
\item For every continuous potential $\pot:X\to(-\infty,0]$ such that $K=\pot^{-1}(0)$ is $T$-invariant and has zero topological entropy, there exist a continuous nonlinearity $F : (-\infty,0]\to (-\infty,0]$ with $F(0)=0$ such that the energy $\En(\mu) = F(\mu(\pot))$ exhibits a freezing phase transition with ground states supported on $K$.
\end{enumerate}
\end{maintheorem}
The first item is not directly related to the \emph{non-linear} thermodynamical formalism, but its analysis is a simple application of the tools developed here (more precisely, we rely on the entropy-potential diagram introduced in Section \ref{sec:convexity} which is central to our non-linear study).


\subsection{Questions}\label{subsec-questions}
We close this introduction with a few {\JB more open} questions. 
\begin{itemize}
 \item Without assuming abundance of ergodicity, does a variational principle hold in \emph{restriction to ergodic} measures, that is:
   $$
        \sup_{\mu\in\Proberg(T)} \Pi^F(T,\pot,\mu) = \Pi^F_\top(T,\pot)?
    $$
(See Remark~\ref{rem-vp-erg}.)
 \item Can one find a subshift of finite type, H\"older-continuous potentials and a real-analytic nonlinearity\footnote{Recall that we ask that {\JB real-analytic} $F$ be defined on an \emph{open} set containing the compact set of all possible values of $(\mu(\pot_1),\dots,\mu(\pot_d))$. This in particular prevents the trivial choice $F(\vec z) = -\sup\{h(T,\mu) \colon (\mu(\pot_1),\dots,\mu(\pot_d))=\vec z \}$.} such that there exist infinitely many nonlinear equilibrium measures? What if we additionally impose the quadratic nonlinearity, i.e., $F(z) = \frac12\lVert z\rVert^2$?
\item
Can one find a ``natural'' energy {\JB (necessarily not an energy with potentials) on  some subshift of finite type} such that the non-linear equilibrium measure is unique but not ergodic?
\item {\JB For the doubling map and the Wasserstein energy $W_p(\cdot,\lambda)$ from Section~\ref{sec-Wasserstein},} is there a finite $\beta>0$ at which $\delta_0$ is an equilibrium measure? Is $\lambda$ still an equilibrium measure for some $\beta>0$? What happens just after $\lambda$ ceases to be an equilibrium?

\end{itemize}

\section{Variational principle}
\newcommand\hQ{{\widehat Q}}

In this section we prove Theorem~\ref{thm-var-prin-NL}. We first introduce some convenient notations. We fix a compact metric space $X$, a map $T:X\to X$ and an energy $\En$. In order to be as general as possible, we do not assume $T$ to be continuous for now, but only Borel-measurable.  Note that $X^n$ being compact, every subset is totally bounded; this ensures the finiteness of $(\eps,n)$-separated sets even when $T$ is not assumed to be continuous.
We often omit $T,\En$ from the notation, i.e., $\NLP_\top = \NLP_\top^\En(T)$, $\NLP(\mu) = \NLP^\En(T,\mu)$ etc.

Recall the definitions of the empirical measures of a point $x\in X$, of the  nonlinear weight of a subset $\mathcal C\subset X$, and of the partition function:
$$
    \Delta_x^n = \frac1N\sum_{k=0}^{n-1} \delta_{T^kx} \qquad
    \omega_n(\mathcal C) :=  \sum_{x\in\mathcal C} e^{n\En(\Delta_x^n)} \qquad
    \zeta(\eps,n) := \sup_{\mathclap{\substack{\mathcal{C} \\ (\eps,n)\text{-separated}}}}\ \omega_n(\mathcal C).
$$

We define for use in this section the following notation:
\begin{align*}
    \NLP_\top(\eps)&=\limsup_{n\to\infty} \frac1n \log \zeta(\eps,n) &
    \text{so that }\NLP_\top &= \lim_{\eps\to 0} \NLP_\top(\eps) \\
    \underline{\NLP}_\top(\eps) &= \liminf_{n\to\infty} \frac1n \log \zeta(\eps,n) &\underline{\NLP}_\top &= \lim_{\eps\to 0} \underline{\NLP}_\top(\eps).
\end{align*}


\subsection{Preliminaries}

We will use  the Wasserstein distance on the set $\Prob$ of probability measures on $X$. Proofs of the statements we need can be found in many places, e.g., \cite{Villani-book}. 

The distance between $\mu_1,\mu_2\in\Prob$ can be defined as
\[ \wass(\mu_1,\mu_2) = \sup \big\{\mu_1(f)-\mu_2(f) \colon f \text{ $1$-Lipschitz function } X\mapsto \mathbb{R} \big\}.\]
The ``Kantorovich duality'' states that this definition is equivalent to
\[ \wass(\mu_1,\mu_2) = \inf \Big\{\int d(x,y) \dd \pi(x,y) \colon \pi\in\Gamma(\mu_1,\mu_2) \Big\} \]
where $d$ is the distance on $X$ and $\Gamma(\mu_1,\mu_2)$ is the set of `transport plans'', i.e., Borel probability measures on $X\times X$ with marginals $\mu_1$ and $\mu_2$.
Moreover in these definitions both the supremum and the infimum are reached; a transport plan realizing the Wasserstein distance is said to be \emph{optimal}. The compactness of $X$ implies that the Wassertein distance induces the weak-star topology on $\Prob$, and that Wasserstein distance can be bounded above by total variation distance:
\[ \wass(\mu_1,\mu_2) \le \diam(X) \lVert \mu_1-\mu_2\rVert_{\mathrm{TV}} \]

We will also use the following reformulation of Birkhoff's ergodic theorem.
\begin{lemma}\label{l:Birkhoff}
Let $\mu\in\Prob(T)$ be ergodic. Then for $\mu$-almost all $x\in X$, we have $\Delta_x^n \to \mu$ in the weak-star topology.
\end{lemma}

\begin{proof}
Let $(f_k)_{k\in\mathbb{N}}$ be a dense sequence of the space $C(X,\RR)$ of continuous functions $X\to\mathbb{R}$, endowed with the uniform norm. There exists a set $E$ {\JB with $\mu(E)=1$} such that for all $x\in E$ and all $k\in\mathbb{N}$, $\Delta_x^n(f_k) \to \mu(f_k)$ as $n\to\infty$.

Let $f\in C(X,\RR)$ and $\eps>0$. There exists $k\in\mathbb{N}$ such that $\lVert f-f_k\rVert_\infty \le \eps$, and $N\in\mathbb{N}$ such that for all $n\ge N$ and all $x\in E$, $\lvert \Delta_x^n(f_k)-\mu(f_k)\rvert\le\eps$. We then have $\lvert \Delta_x^n(f)-\mu(f)\rvert \le 3\eps$.
\end{proof}

\subsection{First part}\label{sec:th-A-1}

Theorem \ref{thm-var-prin-NL} starts with the equalities:
\begin{equation}\label{eq-def-Ptop1}
      \sup_{\mu\in\Prob(T)} \NLP(\mu) = \NLP_\top = \underline{\NLP}_\top.
\end{equation}
We will first prove
$$
     \sup_{\mu\in\Proberg(T)} \NLP(\mu) \underset{\tiny\In\label{In:1}}{\leq} \underline{\NLP}_\top \underset{\tiny\In\label{In:2}}{\leq} \NLP_\top \underset{\tiny\In\label{In:3}}{\leq}
      \sup_{\mu\in\Prob(T)} \NLP(\mu).
$$
Inequality $\Inref{In:1}$  is proved in Proposition \ref{prop:upper-ergodic}. Inequality  $\Inref{In:3}$ is proved in Proposition \ref{prop-lower-bound}. Inequality $\Inref{In:2}$ immediatlely follows from the definitions of $\NLP_\top$ and  $\underline{\NLP}_\top$. 

The missing inequality 
$$
   \sup_{\mu\in\Prob(T)} \NLP(\mu) \underset{\tiny\In\label{In:4}}{\leq}\sup_{\mu\in\Proberg(T)} \NLP(\mu)
$$ 
is proved in Proposition \ref{prop-ine probproberg} 
{\JB assuming an abundance of ergodic measures.}

\begin{remark}\label{rem-vp-erg}
If {\JB $(T,\En)$} is continuous but {\JB without} abundance of ergodicity, 
the following example shows that inequality 
$$\NLP_{\top} <\sup_{\mu\in\Prob(T)} \NLP(\mu)$$
may hold.

 Let $(X,T)$ be the union of two distinct fixed points $p,q$. Let $\En(\mu) = F(\mu(\pot))$ with $F(z)=-z^2$,  $\pot(p)=1$, $\pot(q)=-1$. Then $\NLP(\mu)=0$ for $\mu=\frac12(\delta_p+\delta_q)$ whereas $\NLP_\top=-1$. 
\end{remark}

\subsubsection{Bounding below the nonlinear topological pressure}
We prove Inequality $\Inref{In:1}$, then Inequality $\Inref{In:4}$ 
{\JB assuming an abundance of ergodic measures.}
Note that continuity of $T$ is not needed at that stage.

\begin{proposition}[Inequality \Inref{In:1}]\label{prop:upper-ergodic}
Recall that $X$ is a compact metric space. If $T:X\to X$ is Borel-measurable, then
for all \emph{ergodic} $\mu\in\Prob(T)$, we have $\NLP(\mu)\le \underline{\Pi}_\top$.
\end{proposition}

\begin{proof}
Consider any $\gamma>0$. Since $\En$ is continuous and $\Prob$ is compact, $\En$ is uniformly continuous: there exists $\delta>0$ such that for all $\nu\in\Prob$, $\wass(\nu,\mu)\le 2\delta \implies \En(\nu)\ge \En(\mu)-\gamma$.

By Lemma \ref{l:Birkhoff}, there is a set $A\subset X$ with $\mu(A)\ge\frac34$ and $M_A\in\mathbb{N}$  such that for all $x\in A$ and all $n\ge M_A$ we have
$\wass(\Delta_x^n,\mu)\le \delta$.

By the Brin-Katok entropy formula \cite{Brin-Katok}, there exist $B\subset X$ with $\mu(B)\ge \frac34$ and $M_B\in\mathbb{N}$ such that for all $x\in B$ and all $n\ge M_B$ we have
 $$
    \Big\lvert \frac1n \log \mu(B(x,2\delta,n)) + h(T,\mu) \Big\rvert \le \gamma.
 $$

Consider any $n\ge\max(M_A,M_B)$ and any $0<\varepsilon\le\delta$. Let $\mathcal C$ be any $(\varepsilon,n)$-separated set of $X$ that is maximal with respect to inclusion; in particular, $\mathcal{C}$ is an $(\eps,n)$-cover, hence a $(\delta,n)$-cover. Let $\mathcal C'$ be a minimal subset of $\mathcal C$ that is an $(\delta,n)$-cover of $A\cap B$. 

On the one hand, for all $x\in \mathcal{C}'$ by minimality $B(x,\delta,n)$ intersects $B$; picking any $y$ in the intersection, we get $\mu(B(x,\delta,n)) \le \mu(B(y,2\delta,n)) \le e^{n(\gamma-h(T,\mu))}$. Since $\mu(A\cap B)\ge \frac12$, it follows 
 $$
    \lvert \mathcal{C}'\rvert \ge \frac12 e^{n(h(T,\mu)-\gamma)}.
 $$

On the other hand, for all $x\in \mathcal{C}'$ by minimality $B(x,\delta,n)$ intersects $A$; picking any $y$ in the intersection, we have $\wass(\Delta_y^n,\mu)\le \delta$ and $d(T^i x,T^iy)\le \delta$ for all $i\in \{0,\dots, n-1\}$. By considering the transport plan $\sum_i \frac1n \delta_{T^i x} \otimes \delta_{T^i y}$, we see that $\wass(\Delta_x^n,\Delta_y^n)\le \delta$. The triangular inequality then ensures $\wass(\Delta_x^n,\mu)\le 2\delta$, and we get
 $$
    \En(\Delta_x^n)\ge \En(\mu)-\gamma.
 $$

Using these two inequalities, we get
 $$
    \omega_n(\mathcal{C}) \ge \omega_n(\mathcal{C}') \ge \lvert \mathcal{C}'\rvert \min_{x\in\mathcal{C}'} e^{n\En(\Delta_x^n)} \ge \frac12 e^{n(h(T,\mu)-\gamma)} e^{n(\En(\mu)-\gamma)} \ge \frac12 e^{n(\NLP(\mu)-2\gamma)}.
 $$
Since $\mathcal C$ is $(\eps,n)$-separated, we get
 $$
    \frac1n \log\zeta(\eps,n) \ge \NLP(\mu)-2\gamma -\frac1n \log 2.
 $$
Taking the infimum limit as $n\to\infty$, we obtain that for all $\gamma>0$, there exist $\delta>0$ such that for all $\varepsilon\in(0,\delta)$:
 $$
    \underline{\NLP}_{\top}(\varepsilon)\ge \NLP(\mu)-2\gamma,
 $$
and letting $\eps$ then $\gamma$ go to zero ends the proof.
\end{proof}

Observe that we only used lower-semicontinuity for $\En$ here; but its upper-semicontinuity ensures it reaches its supremum, a desirable feature. This motivates the continuity requirement in the definition of an energy.


\begin{proposition}[Inequality \Inref{In:4}]\label{prop-ine probproberg}
If $T$ is Borel-measurable and
$(T,\En)$ has an abundance of ergodic measures, then 
$\disp \sup_{\mu\in\Prob(T)} \NLP(\mu) \le\sup_{\mu\in\Proberg(T)} \NLP(\mu)$.
\end{proposition}

\begin{proof} Let $\mu$ be any invariant probability measure. 
%
{\JB Since} $(T,\En)$ has an abundance of ergodic measures, there is a sequence of measures $\nu_n\in\Proberg(T)$ such that {\JB $\lim_{\ninf} h(T,\nu_n)+\En(\nu_n) \ge h(T,\mu)+\En(\mu)$}; this yields that 
$$ h(T,\mu)+\En(\mu) \le \sup_{\nu\in\Proberg(T)} \NLP(\nu)$$
holds for every $\mu$ in $\Prob(T)$. 
\end{proof}

\subsubsection{Bounding from above the nonlinear topological pressure: Inequality $\Inref{In:3}$}

To end the proof of equality \eqref{eq-def-Ptop}, it remains to construct measures almost realizing the nonlinear topological pressure. We divide the proof into several lemmas that we shall reuse in Section \ref{sec:Gibbs}. We follow the strategy of Misiurewicz' proof of the linear variational principle, from which we extract the following result.
We recall that $H_{\mu}(\al)$ stands for the entropy for the measure $\mu\in \Prob(T)$  of the partition $\al$.

\begin{lemma}[Misiurewicz \cite{Misiurewicz-VP}]\label{lem-Misiurewicz}
Fix $\eps>0$ and let $(\mathcal{C}_k)_{k\in\mathbb{N}}$ be a sequence of $(\eps,n_k)$-separated sets where $n_k\to\infty$. Assume that for each $k$, $\sigma_k$ is a probability measure concentrated on $\mathcal{C}_k$ and that 
\[\mu_k = \frac{1}{n_k}\sum_{\ell=0}^{n_k-1} T_*^\ell \sigma_k\]
converges in the weak star topology to some measure $\mu_\infty$.

{\JB Fix any finite} partition $\alpha$ of $X$ into subsets of diameter less than $\eps$ and with negligible boundaries with respect to $\mu_\infty$ {\JB (such an $\alpha$ always exists). Then} for 
all $m\in\mathbb{N}$, 
\[ H_{\mu_k}(\alpha^m) \ge \frac{m}{n_k} H_{\sigma_k}(\alpha^{n_k}) -\frac{2m^2}{n_k}\log\lvert\alpha\rvert \qquad \forall k \text{ such that }n_k\ge 2m \]
and $H_{\mu_k}(\alpha^m)\to H_{\mu_\infty}(\alpha^m)$ as $k\to\infty$.
\end{lemma}

The proof is not reproduced here, let us simply mention that it consists in partitioning in $m$ different ways the integer interval $\llbracket 0, n_k-1\rrbracket$ into subintervals of length $m$ plus a small remainder at the start and end. Note that the hypothesis that $\mathcal{C}_k$ is $(\eps,n_k)$-separated is intended to make the computation of $H_{\sigma_k}(\alpha^{\JB n_k})$ a formality: each element of $\alpha^{\JB n_k}$ contains at most one element of $\mathcal{C}_k$.

To address the nonlinearity, we now divide the space of measures into parts where the energy is almost constant, and then split $(\eps,n)$-separated sets according to this partition.

\begin{lemma}\label{lem-dividing}
Let $\eps>0,\gamma\in(0,1)$ and $(\mathcal{C}_k)_{k\in\mathbb{N}}$ be a sequence of $(\eps,n_k)$-separated subsets of $X$ where $n_k\to\infty$. There exist $N=N(\gamma)\in\mathbb{N}$, real numbers $(E_i)_{1\le i\le N}$, a sequence of partitions $\mathbf{D}_k=(\mathcal{D}_{k,i})_{1\le i\le N}$ of $\mathcal{C}_k$ and $I\subset \llbracket 1,N\rrbracket$ such that, up to extracting a subsequence (still denoted by $(n_k)_k$), for all $k$:
\begin{enumerate}
\item $\sum_{i\notin I}\omega_{n_k}(\mathcal{D}_{k,i})\le \gamma \omega_{n_k}(\mathcal{C}_k)$,
\item for all $i\in I$, $\omega_{n_k}(\mathcal{D}_{k,i})\ge \frac{\gamma}{N} \omega_{n_k}(\mathcal{C}_k)$,
\item for all $i$, for all $\mu\in\Prob$ that is a convex combination of the measures $\Delta_x^{n_k}$ where $x$ runs over $\mathcal{D}_{k,i}$, $\lvert \En(\mu)-E_i\rvert \le \gamma$,
\item for all $i\in I$, $\lvert \mathcal{D}_{k,i}\rvert \ge \frac{\gamma}{N} \omega_{n_k}(\mathcal{C}_k)e^{-n_k(E_i+\gamma)}$.
\end{enumerate}
\end{lemma}

\begin{proof}
Since $\En$ is continuous and $\Prob$ is compact, there exists $\delta>0$ such that for all $\mu,\nu\in\Prob$, $\wass(\mu,\nu)\le\delta\implies \lvert \En(\mu)-\En(\nu)\rvert\le \gamma$.

Let $S=\{\sigma_1,\dots,\sigma_N\}$ be a $\delta$-covering of $(\Prob,\wass)$ and set $E_i : =\En(\sigma_i)$. For each $\mu\in\Prob$ we can define $i(\mu) = \min\{i \mid \wass(\mu,\sigma_i)\le\delta\}$. We then set $V_i=\{\mu\in \Prob\mid i(\mu)=i\}$; the $(V_i)$ form a partition of $\Prob$, and for all $\mu\in V_i$ we have $\lvert \En(\mu) - E_i\rvert \le \gamma$.

For all $k,i$, let $\mathcal{D}_{k,i} = \{x\in \mathcal{C}_{k} \mid \Delta_x^{n_k}\in V_i\}$. Up to extracting a subsequence, we can assume that for each $i\in\llbracket 1,N\rrbracket$ , either $\omega_{n_k}(\mathcal{D}_{k,i})\ge\frac{\gamma}{N} \omega_{n_k}(\mathcal{C}_{k})$ for all $k$, or  $\omega_{n_k}(\mathcal{D}_{k,i})<\frac{\gamma}{N} \omega_{n_k}(\mathcal{C}_{k})$ for all $k$. Let $I$ be the set of indices $i$ belonging to the first category. We have obtained the first two items; note that for any $k$ we have $\omega_{n_k}(\mathcal{C}_{k}) = \sum_i \omega_{n_k}(\mathcal{D}_{k,i})$, so that $I$ must be non empty. 

Consider a probability measure $\mu= \sum_{x\in\mathcal{D}_{k,i}} a_x \Delta_x^{n_k}$; then  $\wass(\mu,\sigma_i)\le\delta$: indeed, we have for each $x\in  \mathcal{D}_{k,i}$ a coupling $\pi_x\in\Gamma(\Delta_x^{n_k},\sigma_i)$ of cost at most $\delta$, and the cost of the coupling $\sum_x a_x \pi_x\in\Gamma(\mu ,\sigma_i)$ is thus at most $\delta$. As a consequence, $\lvert \En(\mu)-E_i\rvert \le \gamma$.

Given $i\in I$, combining both previous items yields:
\[\frac{\gamma}{N}\omega_{n_k}(\mathcal{C}_k) \le \omega_{n_k}(\mathcal{D}_{k,i}) \le \lvert \mathcal{D}_{k,i}\rvert e^{n_k(E_i+\gamma)} \] 
so that $\lvert \mathcal{D}_{k,i} \rvert\ge \frac{\gamma}{N}\omega_{n_k}(\mathcal{C}_k)e^{-n_k(E_i+\gamma)}$.
\end{proof}

\begin{lemma}\label{lem-clusterpoints}
Using the notations of the previous lemma, fix any $i\in I$ and assume $\frac{\log(\omega_{n_k}(\mathcal{C}_k))}{n_k}$ converges as $k\to+\8$ (this induces no loss in generality, since we already extracted subsequences and can do it once more). Define a sequence of probability measures by
\[\tilde\mu_k = \frac{1}{\omega_{n_k}(\mathcal{D}_{k,i})} \sum_{x\in\mathcal{D}_{k,i}} e^{n_k\En(\Delta_x^{n_k})} \Delta_x^{n_k}\]
If $T$ is continuous, then any accumulation points $\tilde\mu_\infty$ of this sequence is $T$-invariant and satisfies $\NLP(\tilde\mu_\infty)\ge \lim_{k} \frac{\log(\omega_{n_k}(\mathcal{C}_k))}{n_k} -5\gamma$.
\end{lemma}

The sequence given by
\[\mu_k = \frac{1}{\lvert \mathcal{D}_{k,i} \rvert} \sum_{x\in\mathcal{D}_{k,i}} \Delta_x^{n_k}\]
could be preferred to $(\tilde\mu_k)_k$ for the proof of \Inref{In:3}, and can be treated in pretty much the same way. However, we will need $(\tilde\mu_k)_k$ in Section \ref{sec:Gibbs} to describe the accumulation points of Gibbs ensembles.

\begin{proof}
Let first $\tilde\mu_\infty$ be an accumulation point of $(\tilde\mu_k)$; up to extracting a further subsequence, we assume $\tilde\mu_\infty=\lim_k \tilde\mu_k$.

To check that $\tilde\mu_\infty\in\Prob(T)$, first observe that $\wass(\Delta_x^{n_k},T_*\Delta_x^{n_k})\le \frac{\diam X}{n_k}$ by the total variation bound (i.e., using a coupling that leaves the common part $\sum_{1\le j<n_k} \delta_{T^j x}$ in place and moves the remaining mass $\frac{1}{n_k}$ from $x$ to $T^{n_k}x$) and conclude using an averaged coupling as in the proof of Lemma \ref{lem-dividing} above that $\wass(\tilde\mu_k,T_*\tilde\mu_k)\to 0$. Up to this point, no use was made of the continuity assumption on $T$. But we want to pass to the limit in the arguments of $\wass$, and the continuity ensures that $T_*\tilde\mu_k\to T_*\tilde\mu_\infty$. Then we get $\wass(\tilde\mu_\infty,T_*\tilde\mu_\infty)=0$, and thus $\tilde\mu_\infty\in\Prob(T)$. Note also that $\En(\tilde\mu_k)\le E_i+\gamma$ for all $k$, so that the same holds for $\tilde\mu_\infty$.

Consider a partition $\alpha$ of $X$ whose element have diameter at most $\eps$ and whose boundaries have zero measure with respect to $\tilde\mu_\infty$. Setting
\[\sigma_k = \frac{1}{\omega_{n_k}(\mathcal{D}_{k,i})} \sum_{x\in\mathcal{D}_{k,i}} e^{n_k \En(\Delta_x^{n_k})}\delta_x\]
we have $\tilde\mu_k= \frac{1}{n_k} \sum_{j=0}^{n_k-1} T_*^j \sigma_k$ and, since $\mathcal{D}_{k,i}$ is $(\eps,n_k)$-separated,
 \begin{align*}
    H_{\sigma_k}(\alpha^{n_k}) &= \sum_{x\in\mathcal{D}_{k,i}} p_x \log\frac{1}{p_x} &\text{where } p_x &= \frac{e^{n_k \En(\Delta_x^{n_k})}}{\omega_{n_k}(\mathcal{D}_{k,i})}\\
    &&\log \frac{1}{p_x} &\ge \log\big(\lvert\mathcal{D}_{k,i}\rvert e^{n_k(E_i-\gamma)} \big) - n_k(E_i+\gamma) \\
    H_{\sigma_k}(\alpha^{n_k}) &\ge \log\lvert\mathcal{D}_{k,i}\rvert-2n_k\gamma.
 \end{align*}
Lemma \ref{lem-Misiurewicz} applied to $\mathcal{D}_{k,i}$ asserts that 
$H_{\tilde\mu_k}(\alpha^{m})\ge \frac{m}{n_k} \log \lvert \mathcal{D}_{k,i}\rvert -2\gamma m-\frac{2m^2}{n_k}\log\lvert\alpha\rvert$ for all $m\in\mathbb{N}$ and all $k$ such that $n_k\ge 2m$. It follows that for all $m$ and all $k$ large enough (then taking successive limits as $k\to\infty$ and $m\to\infty$):
\begin{align*}
\frac1m H_{\tilde\mu_k}(\alpha^m) &\ge \frac{1}{n_k} \log \lvert \mathcal{D}_{k,i}\rvert -3\gamma \\
  &\ge \frac{\log \omega(\mathcal{C}_k)}{n_k}-E_i-4\gamma-\frac{\log(N/\gamma)}{n_k} \\
\frac1m H_{\tilde\mu_\infty}(\alpha^m)  &\ge \lim_k \frac{\log \omega(\mathcal{C}_k)}{n_k}-E_i-4\gamma \\
h(T,\tilde\mu_\infty) &\ge \lim_k \frac{\log \omega(\mathcal{C}_k)}{n_k} - \En(\mu_\infty)-5\gamma \\
\NLP(\mu_\infty) &\ge \lim_k \frac{\log \omega(\mathcal{C}_k)}{n_k} -5\gamma.
\end{align*}
\end{proof}

\begin{proposition}\label{prop-lower-bound}
If $T$ is continuous, then we have $\displaystyle\sup_{\mu\in\Prob(T)} \NLP(\mu) \ge \NLP_\top$.
\end{proposition}

\begin{proof}
Let $\gamma>0$, and choose $\varepsilon>0$ small enough to ensure 
\[\NLP_\top(\varepsilon):=\limsup_{n\to\infty} \frac1n \log \zeta(\eps,n)\ge \NLP_\top-\gamma.\]

For each $n\in\mathbb{N}$, let $\mathcal{C}_n$ be an $(\eps,n)$-separated subset of $X$ realizing $\zeta(\eps,n)$. Let $(n_k)_k$ be a sequence of integers such that $n_k\to\infty$ and $\frac{1}{n_k}\log \zeta(\eps,n_k) \to \NLP_\top(\varepsilon)$. 

We apply Lemma \ref{lem-dividing}, fix any $i\in I$, define $\tilde\mu_k$ as in Lemma \ref{lem-clusterpoints}  and let $\tilde\mu_\infty$ be any of its accumulation points. We then have
\[\NLP(\tilde\mu_\infty) \ge \lim \frac{\log\omega_{n_k}(\mathcal{C}_k)}{n_k}-5\gamma = \NLP_\top(\eps)-5\gamma \ge \NLP_\top-6\gamma.\]
Letting $\gamma$ go to zero ends the proof.
\end{proof}

Assuming $T$ is continuous and {\JB abundance of ergodic measures}, 
we have shown that:
 $$
    \NLP_\top \le \sup_{\mu\in\Prob(T)} \NLP(\mu) \leq \underline{\NLP}_\top .
 $$
Since, obviously, $\underline{\NLP}_\top \le \NLP_\top$, the above inequalities must be equalities. This proves eq.~\eqref{eq-def-Ptop1} under the assumptions of Theorem \ref{thm-var-prin-NL}.


\subsection{Proof of Theorem A: the expansive case}

We assume that $T$ is a homeomorphism admitting the expansivity constant $\eps_0>0$. To begin with, we let $0<\eps\le\eps_0$ and  show that
 \begin{equation}\label{eq-Ptop-exp0}
    \NLP_\top(\eps)=\NLP_\top(\eps_0):=\limsup_{n\to\infty} \frac1n\log\zeta(\eps_0,n).
 \end{equation}  

Let us prove that $\NLP_\top(\eps)\leq\NLP_\top(\eps_0)$ by extracting an $(\eps_0,n)$-separated set from an $(\eps,n)$-separated one and comparing their weights.

We first fix $\gamma>0$ arbitrarily small. By the uniform continuity of $\mathcal E$ on $\Prob$,
there is $0<\delta \leq2\eps$ such that
 \begin{equation}\label{eq-ECU}
   W(\mu,\nu)<\delta\implies \lvert \mathcal E(\mu)-\mathcal E(\nu) \rvert<\gamma. 
 \end{equation}
 
We need the following version of the Theorem of uniform expansivity.
 \begin{claim}
There exists $N\geq1$ such that  for all $n\geq 2N$, for any $x\in X$,
 \begin{equation}\label{eq-expansivity}
  \forall N\leq k<n-N\quad \diam(T^k(B(x,\eps_0,n)))<\delta/2 \leq\eps.
 \end{equation}
\end{claim}
\begin{proof}[Proof of the Claim]
If this does not hold, pick for every $N${\BK :} $n_{N}\ge 2N$, $N\le k_{N}\le n_{N}-N$ and $x_{N}$ such that 
$$\diam(T^{k_{N}}(B(x_{N},\eps_{0},n_{N})))\ge \delta/2.$$
Pick $N_{0}$ and $N\ge N_{0}$. 
Note the following inclusions:
\begin{multline*}
  B(T^{k_{N}}(x_{N}),\eps_{0}, N_{0})\supset B(T^{k_{N}}(x_{N}),\eps_{0}, N))\\ 
  \supset B(T^{k_{N}}(x_{N}),\eps_{0}, n_{N}-N)\supset T^{k_{N}}(B(x_{N},\eps_{0},n_{N})).
\end{multline*}

Then, consider any accumulation point $y$ for $y_{N}:=T^{k_{N}}(x_{N})$. This yields 
$$\forall N_{0},\ \diam(B(y,\eps_{0},N_{0}))\ge \delta/2.$$
This is in contraction with the fact that $\eps_{0}$ is an expansivity constant. 
\end{proof}


We now fix some finite $(\eps/2,N)$-cover $C_\eps$ of $X$ and some large enough integer $n\geq1$ (exactly how large will be specified later on; in particular we assume equation \eqref{eq-expansivity} holds).

Given an arbitrary nonempty $(\eps,n)$-separated subset $S$ of $X$, we consider $\hS$ any  $(\eps_0,n)$-separated subset of $S$, maximal for inclusion. 

\begin{claim}\label{claim-W}
The following facts hold:
 \begin{enumerate}
  \item\label{itemNotEmpty} For every $x\in S$, $B(x,\eps_0,n)\cap \hS$ is nonempty;
  \item\label{itemCloseEnergy}  For every $x\in S$ and every $y\in B(x,\eps_0,n)$, $\lvert \En(\Delta^n_x)-\En(\Delta^n_y) \rvert \leq \gamma $.
  \item\label{itemBounded} For every $y\in \hS$, $1\le \lvert B(y,\eps_0,n)\cap S \rvert \le \lvert C_\eps\rvert^2$;
 \end{enumerate}
\end{claim}

\begin{proof}[Proof of the claim]
To see that (\ref{itemNotEmpty}) holds, note that, if for some $x\in S$, $B(x,\eps_0,n)\cap \hS=\emptyset$, $\hS\cup\{x\}$ would still be $(\eps_0,n)$-separated, contradicting the maximality of $\hS$. 

To prove (\ref{itemCloseEnergy}), let $x,y$ be any two points of $X$ with $y\in B(x,\eps_0,n)$. By eq.~\eqref{eq-expansivity}, $d(T^kx,T^ky)<\delta/2$ for all $N\leq k<n-N$, hence we get:
 $$\begin{aligned}
    W(\Delta^n_x,\Delta^n_y) 
      &\leq \frac1n\sum_{k=0}^{n-1} d(T^kx,T^ky) 
      \leq \frac{2N}{n}\diam(X) + \frac\delta2 <\delta
 \end{aligned}$$
 for large enough $n$. The claim (\ref{itemCloseEnergy}) now follows from eq. \eqref{eq-ECU}.

We turn to (\ref{itemBounded}). Since $\hS\subset S$, $y\in B(y,\eps_0,n)\cap S$ so this set is not empty.  To prove the upper bound let
$I:B(y,\eps_0,n)\cap S\to C_\eps\times C_\eps$ satisfy $I(z)=(w,w')$ with $w\in B(z,\eps/2,N)$ and $w'\in B(T^{n-N}z,\eps/2,N)$. Observe that such a map exists since $C_\eps$ is a $(\eps/2,N)$-cover of $X$ and let us check that $I$ is injective. Indeed, let $z,z'\in B(y,\eps_0,n)\cap S$ with $I(z)=I(z')=:(w,w')$ and note:
 \begin{itemize}
  \item for all $0\leq k<N$, $d(T^kz,T^kz')\leq d(T^kz,T^kw)+d(T^kw,T^kz') <\eps$;
  \item for all $N\leq k<n-N$, $d(T^kz,T^kz') < \eps$ from eq.~\eqref{eq-expansivity};
  \item for all $n-N\leq k<n$,
   $$d(T^kz,T^kz')\leq d(T^kz,T^{k-(n-N)}w')+d(T^{k-(n-N)}w',T^kz') <\eps.$$
 \end{itemize}
Thus $z,z'\in S$ are not $(\eps,n)$-separated and thus must be equal, proving the injectivity of the map $I$, proving (\ref{itemBounded}).
Claim \ref{claim-W} is established.
\end{proof}

We now compare the weights of $S$ and $\hS$:
 $$\begin{aligned}
   \omega_n(\hS) &= \sum_{y\in\hS} e^{n\En(\Delta^n_y)}
     \geq \sum_{y\in\hS} \min_{x\in B(y,\eps_0,n)\cap S} e^{n\En(\Delta^n_x)}
          &\text{since }\hS\subset S\\
    &\geq \sum_{y\in\hS} \frac{e^{-\gamma n} }{\lvert B(y,\eps_0,n)\cap S\rvert} \sum_{x\in B(y,\eps_0,n)\cap S} e^{n\En(\Delta^n_x)} &\text{ by eq. \eqref{itemCloseEnergy}}\\
     &\geq \sum_{y\in\hS} \frac{e^{-\gamma n}}{\lvert C_\eps\rvert^2} \sum_{x\in B(y,\eps_0,n)\cap S}  e^{n\En(\Delta^n_x)} &\text{ by eq. \eqref{itemBounded}}\\
    &\geq \frac{e^{-\gamma n}}{\lvert C_\eps\rvert^2} \sum_{x\in S} \lvert B(x,\eps_0,n)\cap\hS\rvert \,  e^{n\En(\Delta^n_x)} & \text{exchanging the sums}\\
    &\geq  \frac{e^{-\gamma n}}{\lvert C_\eps\rvert^2} \omega_n(S) &\text{ by eq. \eqref{itemNotEmpty}}.
 \end{aligned}$$
Therefore, $\frac1n\log\zeta(\eps_0,n)\geq \frac1n\log\omega_n(\hS)\geq\frac1n\log\zeta(\eps,n)-\gamma-\frac1n\log\lvert C_\eps\rvert^2$. Hence,
 $$
   \Pi_\top(\eps):=\limsup_{n\to\infty}\frac1n\log\zeta(\eps,n)\leq \limsup_{n\to\infty}\frac1n\log\zeta(\eps,n)+\gamma=:\Pi_\top(\eps_0)+\gamma
 $$
as  $\gamma>0$ was arbitrary we obtain: $\NLP_\top(\eps)\leq\NLP_\top(\eps_0)$ for all $0<\eps\leq\eps_0$. The definitions immediately yield the inequality $\zeta(\eps_0,n)\leq\zeta(\eps,n)$ and therefore $\NLP_\top(\eps)=\NLP_\top(\eps_0)$. This proves \eqref{eq-Ptop-exp0}.

The same argument applies to $\underline{\NLP}_\top(\eps):=\liminf_{n\to\infty}\frac1n\log\zeta(\eps,n)$ yielding: $\underline{\NLP}_\top(\eps) = \underline{\NLP}_\top(\eps_0)$. By eq.~\eqref{eq-def-Ptop1}, $\lim_{\eps\to0} \NLP_\top(\eps)= \lim_{\eps\to0} \underline{\NLP} _\top(\eps)$. Thus, $\NLP_\top(\eps)= \underline{\NLP} _\top(\eps)$ for all $0<\eps\leq\eps_0$: the upper and lower limits of $\frac1n\log\zeta(\eps,n)$ as $n$ goes to $\infty$ coincide. Thus, we have a true limit, independently of $\eps\in(0,\eps_0)$:
 $$
   \NLP_\top = \lim_{n\to\infty} \frac1n\log\zeta(\eps,n),
 $$
concluding the proof of Theorem~\ref{thm-var-prin-NL}.


\section{Existence of an equilibrium measure and convergence of the Gibbs ensembles}\label{sec:Gibbs}

In this section we prove Theorem~\ref{thm-existence-NL}. Its existence claim is a simple consequence of the variational principle we just established as Theorem~\ref{thm-var-prin-NL}.

\begin{lemma}\label{lemm-existence}
Assume that $T$ is continuous with $\mu\mapsto h(T,\mu)$ upper semi\-conti\-nuous, and 
{\JB that $(T,\En)$ has an abundance of ergodic measures}. Then the set $\EM$ of nonlinear equilibrium measures is non-empty.

Moreover, for all $\gamma>0$ there exists $\delta>0$ such that invariant measures achieving $\NLP^\En_\top(T)$ up to $\delta$ are $\gamma$-close to $\EM$: for all $\mu\in\Prob(T)$ such that $h(T,\mu)+\En(\mu) > \NLP^\En_\top(T)-\delta$, there exist $\mu'\in\EM$ such that $\wass(\mu,\mu') < \gamma$.
\end{lemma}

\begin{proof}
By assumption $\mu\mapsto h(T,\mu)+\En(\mu)$ is upper semi-continuous on the compact set $\Prob(T)$, it must therefore reach its maximum, which by Theorem~\ref{thm-var-prin-NL} is $\NLP^\En_\top(T)$. $\EM$ must thus be non-empty, and compact.

Given $\gamma>0$, the upper semi-continuous function $\mu\mapsto h(T,\mu)+\En(\mu)$ also reaches its maximum $F_\gamma$ on the compact set $\{\mu\in\Prob(T) \colon \wass(\mu,\EM) \ge\gamma\}$. Since this set is disjoint from $\EM$, $F_\gamma<\NLP^\En_\top(T)$. The positive number $\delta = \NLP^\En_\top(T)-F_\gamma$ has the desired property.
\end{proof}

The second part of Theorem \ref{thm-existence-NL} is proven along the same lines than Proposition~\ref{prop-lower-bound}.
\begin{proposition}
Assume that $T$ is an expansive homeomorphism with $\eps_0>0$ an expansivity constant, and 
{\JB that  $(T,\En)$ has an abundance of ergodic measures}. Let $\nu$ be an accumulation point of $(\eps_0,n)$-Gibbs ensembles as $n\to\infty$. Then $\nu$ can be approximated in the weak-star topology by linear combinations of nonlinear equilibrium measures.
\end{proposition}

\begin{proof}
Note that $T$ being an expansive homeomorphism, entropy is upper semi-continuous.
By the second half of Lemma \ref{lemm-existence}, we are reduced to approximate $\nu$ by convex combination of measures that almost achieve the nonlinear topological pressure.

By definition $\nu$ is the limit of measures of the form
 $$
     \mu_k = \frac{1}{\zeta(\eps_0,n_k)} \sum_{x\in\mathcal{C}_k} e^{n_k\En(\Delta_x^{n_k})} \Delta_x^{n_k}
 $$
where $n_k\to\infty$, $\mathcal{C}_k$ are $(\eps_0,n_k)$-separated sets with $\omega_{n_k}(\mathcal{C}_k) = \zeta(\eps_0,n_k)$. Fix some $\gamma>0$.

Apply Lemma \ref{lem-dividing}, providing $N\in\mathbb{N}$, partitions $\mathcal{D}_{k,1}, \dots ,\mathcal{D}_{k,N}$ of each $\mathcal{C}_k$ and $I\subset \llbracket 1,N\rrbracket$ such that up to further extracting a subsequence (still denoted by $(n_k)_k$), $\sum_{i\notin I} \omega_{n_k}(\mathcal{D}_{k,i}) <\gamma \omega_{n_k}(\mathcal{C}_k)$ and for all $i\in I$, Lemma \ref{lem-clusterpoints} applies.

For each $i$, consider
 $$
    \tilde\mu^i_k = \frac{1}{\omega_{n_k}(\mathcal{D}_{k,i})} \sum_{x\in \mathcal{D}_{k,i}} e^{n_k\En(\Delta_x^{n_k})} \Delta_x^{n_k}
 $$
and assume, up to further extraction, that it converges as $k\to\infty$ to some $\tilde\mu^i$. Then by Lemma \ref{lem-clusterpoints}, whenever $i\in I$:
 $$
    \NLP(\tilde \mu^i)\ge \lim \frac{\log(\omega_{n_k}(\mathcal{C}_k))}{n_k}-5\gamma = \NLP_\top-5\gamma.
 $$
The $\tilde\mu^i$ with $i\in I$ are the almost equilibrium measures we are looking for. 

We have
 $\mu_k = \sum_{i=1}^N a^i_k \tilde\mu^i_k$ where $a^i_k = \frac{\omega_{n_k}(\mathcal{D}_{k,i})}{\omega_{n_k}(\mathcal{C})}$. Up to a further extraction, we can assume that for each $i$ the sequence $(a^i_k)_k$ converges to some number $a_i\in[0,1]$. It follows that   
 $$
    \nu=\sum_{i=0}^N a_i \tilde\mu_i = \sum_{i\in I} a_i \tilde\mu_i + \sum_{i\notin I} a_i \tilde\mu_i
 $$ 
Note that $\sum_{i\notin I} a_i \le \gamma$, i.e., the second term above has total mass less than $\gamma$. For each $i\in I$ we set $b_i = a_i/\sum_{i\in I} a_i$, so that $\tilde\mu = \sum_{i\in I} b_i\tilde\mu_i$ is a convex combination of almost equilibrium states, and by the total variation bound $\wass(\nu,\tilde\mu) = O(\gamma)$.
\end{proof}

Theorem~\ref{thm-existence-NL} is established.

\section{Convexity and nonlinear equilibrium measures}\label{sec:convexity}

In this section, independently of Sections \ref{sec:th-A-1} and \ref{sec:Gibbs}, we prove an extended version of Theorem \ref{thm-equilibria}, i.e., we study the nonlinear formalism for an energy with potentials.  Specifically, we consider a continuous map $T:X\to X$ with finite entropy $h_\top(T)<\infty$ together with an energy defined as
 $$
    \En(\mu) = F\big(\mu(\pot_1),\dots,\mu(\pot_d)\big)
 $$
for all $\mu\in\Prob(T)$ where,  for some positive integer $d$,
 \begin{itemize}
  \item $\pot_1,\dots,\pot_d:X\to\mathbb{R}$ are continuous functions called the \emph{potentials};
  \item $F:U\to\mathbb{R}$ is a smooth function called the \emph{nonlinearity}.
\end{itemize}

Here we assume that $U\subset\mathbb{R}^d$ is an open set containing the compact \emph{rotation set}
 $$
   \rs(\pot_1,\dots,\pot_d) := \big\{\big(\mu(\pot_1),\dots,\mu(\pot_d)\big) \colon \mu\in\Prob(T) \big\}.
 $$
It will sometimes be convenient to write the potentials as a single vector-valued function: $\vpot := (\pot_1,\dots,\pot_d)$, $\vpot(x) := (\pot_1(x),\dots,\pot_d(x))$, $\mu(\vpot):=(\mu(\pot_1),\dots,\mu(\pot_d))$, etc.

We are going to study the nonlinear equilibrium measures:
 $$
    \EM := \big\{\mu\in\Prob(T):h(\mu)+F(\mu(\vpot))\text{ is maximal}\big\}
 $$

\begin{remark}
If one would like to apply our general results (the variational principle of Theorem~\ref{thm-var-prin-NL} and the equidistribution of Gibbs ensembles of Theorem~\ref{thm-existence-NL}), then one should demand $\En(\mu)$ to be defined for all (non-necessarily invariant) probability measures, i.e., the open set $U$ should contain the convex hull of $\{\vpot(x):x\in X\}$.
\end{remark}

The rest of this section is divided as follows. First, we introduce a ``fully nonlinear formalism'' which is the natural setting of our technique and describe the entropy-potential diagram which is a useful visualization. Second we recall the relevant background concerning Legendre duality and we set up appropriate definitions to use this duality and we provide examples of dynamical system satisfying them. Thirdly we weave all this together and apply Legendre duality in the dynamical context to reach the main goal of this section, Theorem \ref{thm:equilibria-G} (which contains Theorem \ref{thm-equilibria}). 
Finally we deduce some uniqueness results (Corollary~\ref{cor-finite-1D}, Propositions \ref{p-generic-unique}~and~\ref{prop-flexibility}).

\subsection{Fully nonlinear pressure}
Our approach applies to the following more general setting:

\begin{definition}\label{def-fully-nonlin}
Given a continuous system $T$ with potentials $\vpot$, a \emph{fully nonlinear pressure} is a function
 \begin{equation}
    \FNLP^G(\mu,\vpot) := G\big(h(\mu); \mu(\pot_1),\dots,\mu(\pot_d)\big)
 \label{eq-G}
 \end{equation}
defined for all $\mu\in\Prob(T)$  {\JB by} some smooth $G:V \to\mathbb{R}$
assumed to be \emph{admissible}: it is defined on an open subset $V$ of $\RR\times\mathbb{R}^{d}$ and satisfies:\footnote{The notation $\partial_0G$ refers to $\partial G/\partial z_{0}$, the derivative with respect to the first variable, corresponding to entropy since the coordinates are numbered as $(z_0,z_1,\dots,z_d)$.}

 $$
  \partial_0G>0
  \text{ and }
   V\supset\{(h(T,\mu),\mu(\pot_1),\dots,\mu(\pot_d)):\mu\in\Prob(T)\}.
  $$
The corresponding set of \emph{fully nonlinear equilibrium measures} is then:
 $$
   \EM(T,G,\vpot) := \{\mu\in\Prob(T):\FNLP^G(\mu,\vpot)\text{ is maximal }\}.
 $$
\end{definition}

We will reduce the problem of maximizing $\FNLP^G$ to the classical, linear thermodynamical formalism by justifying the following claims:
 \begin{itemize}
  \item[(*)] given ${z}\in\rho(\vpot)$, maximizing $\FNLP^G$ and maximizing the linear pressure over
    $$
      \mathcal M( z):=\{\mu\in\Prob(T):\mu(\vpot)= z\}
     $$ 
are both equivalent to maximizing the entropy there;
  \item[(**)] the values $z=\mu(\vpot)$ realized by fully nonlinear equilibrium measures $\mu$ belong to the interior of rotation set $\rs(\vpot)$;
  \item[(***)] there is a diffeomorphism $\operatorname{int}(\rho(\vpot))\to\RR^d$,  ${z}\mapsto{y}$, such that, for every ${z}\in\operatorname{int}(\rho(\vpot))$, there is a linear equilibrium measure $\nu_{\BK y}$ for the potential
   $$
     {y}\cdot\vpot:=\sum_j y_j\pot_j
  $$
with $\nu_y(\vpot)= z$.
 \end{itemize}
The first point is immediate given the assumption that $\partial_0 G>0$. The second and third point will follow from some convex analysis; the second point more precisely follows from the assumption that the gradient of entropy diverges at the boundary in the definition of $C^r$ Legendre systems (Definitions \ref{defi-regular} and \ref{defi-Legendre}) , see the proof of Theorem \ref{thm:equilibria-G}.

\subsection{The entropy-potential diagram and the entropy function}

In light of the above remark (*), we will use the following geometric viewpoint.  The \emph{entropy-potential diagram}, illustrated by Figure \ref{fig:diagramme2D}, is the set
 $$
    \mathcal{D} = \big\{ (z_0;z_1,\dots,z_d) \in [0,+\infty) \times \mathbb{R}^d \colon \exists \mu\in \Prob(T),\; h(T,\mu)\ge z_0,\;\forall i,\;\mu(\pot_i)=z_i\big\},
 $$
$\mathcal{D}$ can be seen as the hypograph of the entropy function, see {\JB the} function $\hf$ below.


\begin{figure}
\centering
\includegraphics[scale=1]{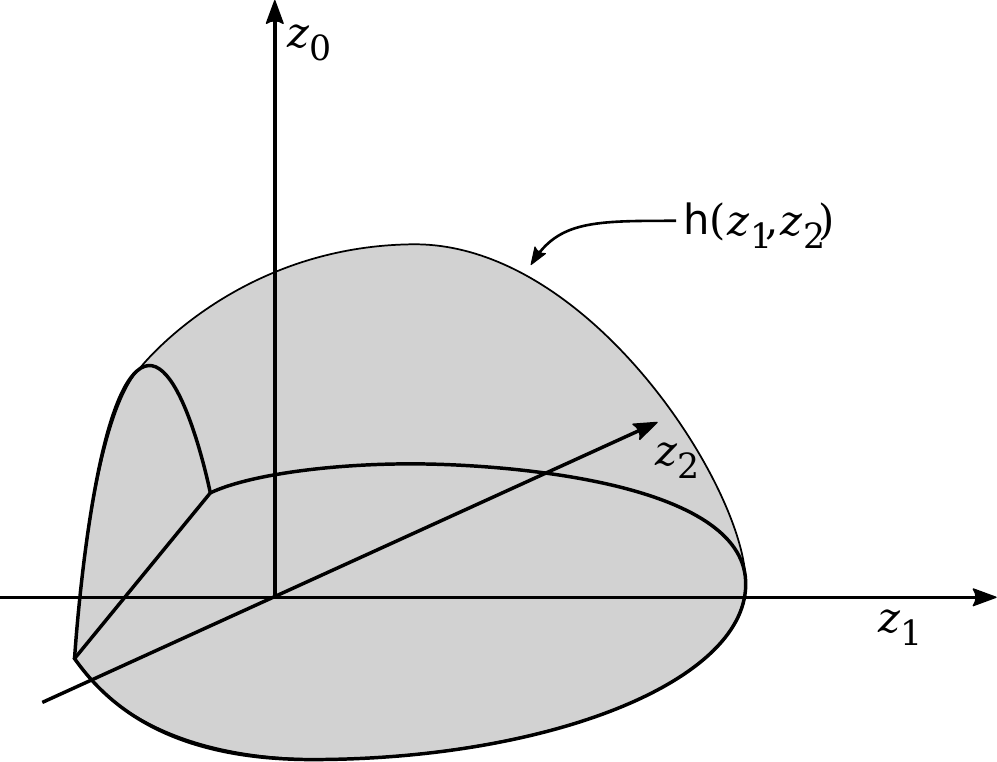}
\caption{An entropy-potentials diagram in two dimensions (first coordinate represented by the vertical axis), in a case when the rotation set is not strictly convex.}\label{fig:diagramme2D}
\end{figure}

Since the Kolmogorov-Sinai entropy is affine, $\mathcal{D}$ is a convex set, and the linear pressure associated to any linear combination  $\sum_i y_i \pot_i$ can be recovered from $\mathcal{D}$ by finding the unique\footnote{Since we fix the normal vector, uniqueness here does not depend on smoothness of $\mathcal{D}$; it is the contact points that may be non-unique, if strict convexity is not assumed.}
support hyperplane with normal vector $(1;y_1,\dots,y_d)$; this has important consequences, see Proposition \ref{prop:nlem}. 
Note that the convexity of $\mathcal{D}$ translates into the concavity of the following function, which is finite exactly on $\rs(\vpot)$:

\begin{definition}\label{defi-entropy}
Given a continuous dynamical system $T$ with potentials $\vpot$, the \emph{(finite-dimensional) entropy function} $\hf:\mathbb{R}^d\to\mathbb{R}\cup\{-\infty\}$ is\footnote{The usual convention $\sup(\varnothing) = -\infty$ is understood.}
 $$
    \hf( z) := \sup_{\mu\in \mathcal{M}( z)} h(T,\mu). 
 $$
\end{definition}

Under our standing assumptions ($X$ compact, $\vpot$ continuous, and $h_\top(T)<\infty$), we have $\{z\in\RR^d:\hf(z)\ne-\infty\}=\rs(\vpot)$.

\begin{remark}
To find the largest value of $\hf+F$ is to find the largest $k$ such that there exists $ z\in\rs(\vpot)$ at which $\hf( z) = -F( z)+k$, i.e., to find the highest vertical translate of the graph of $-F$ that touches the entropy-potential diagram. This makes apparent that the nonlinear equilibrium measures will correspond to linear equilibrium measures associated to one or several linear combinations of potentials, whose coefficients are given by the equations of the tangent hyperplanes at the touching points, see e.g., figure \ref{fig:CurieWeiss}.
\end{remark}

\subsection{Legendre duality}

To apply the well-rounded theory of Legendre duality, let us introduce its classical assumptions, following \cite{Rockafellar-book}.

Recall that the \emph{Legendre transform} $f^*$ of a convex function $f:\RR^d\to\RR\cup\{-\infty\}$ is the convex function:
 $$
  f^*:\RR^d\to\RR\cup\{\infty\},\quad y \mapsto \sup_{x\in\RR^d} \langle  y; x\rangle - f(x).
 $$
If $g$ is a concave function, we set 
 $$
   g^{\#}:\RR^d\to\RR\cup\{\infty\},\quad y \mapsto \sup_{x\in\RR^d} \langle y; x\rangle + g(x),
   $$ i.e., $g^{\#}:=(-g)^*$, which is convex.\footnote{Sometimes, the Legendre transform of a concave function is defined as $-(-g)^*$ instead, so that it is again concave.} 

\medbreak

We will use two classical duality results from \cite{Rockafellar-book}. They ensure that the Legendre transform is an involution on suitable classes of semicontinuous or smooth convex functions. 

\medbreak

\subsubsection*{Semicontinuous functions} A function is \emph{proper} if it is finite at least at one point.

\begin{theorem}\label{theo-duality-CO}
The Legendre transform maps bijectively the class of upper semicontinuous,\footnote{In \cite{Rockafellar-book}, lower semicontinuous convex functions are called \emph{closed}.} proper concave functions to the class of lower semicontinuous proper convex functions. Moreover, this restriction of the Legendre transform is an involution up to sign: for all such $f$, $f=-(f^{\#})^*$.
\end{theorem}

The above theorem implies that  the Legendre transform is an involution over the class of lower semicontinuous proper convex functions $g$ : $(g^*)^*=g$.

\subsubsection*{Smooth functions}
We consider the smoothness classes $C^r$ for $1\leq r\leq\omega$, i.e., for any positive integer $r$ as well as $r=\infty$ (infinitely differentiable) and $r=\omega$ (real-analytic). The following abuses of notation will be convenient: for $r=\infty$ or $\omega$, $C^{r-1}$ just means $C^r$; for $r=0$, a $C^r$ diffeomorphism is a homeomorphism.

\begin{definition}\label{defi-Legendre}
Let $f:\mathbb{R}^d\to\mathbb{R}\cup\{-\infty\}$ be a function. Its \emph{(effective) domain} is the set of points $\dom(f)$ in $\mathbb{R}^d$ where it takes a finite value: $\dom(f)=f^{-1}(\mathbb{R})$. For $1\leq r\leq\omega$,
the function $f$ is said to be \emph{concave of $C^r$ Legendre type} when the following conditions are satisfied:
\begin{enumerate}
\item\label{item-c0} the function $f$ is upper semicontinuous and concave;
\item\label{item-c1} the interior $\inter\dom(f)$ is not empty
and, on this set, $f$ is strictly concave and $C^r$ smooth; when $r\geq2$, we additionally ask that the Hessian of $f$ is everywhere negative definite;
\item\label{item-c2} for all sequences $(x_i)_{i\in\mathbb{N}}$ with $x_i\in\inter(\dom(f))$ which converge to a boundary point of $\dom(f)$,
 $$
    \lim_i \lvert \grad f(x_i) \rvert = +\infty.
 $$
\end{enumerate}
We say that a function $g:\RR^d\to\RR\cup\{\infty\}$ is convex of $C^r$ Legendre type if $-g$ is concave of $C^r$ Legendre type.
\end{definition}

Note that functions are convex of $C^1$ Legendre type exactly when they are convex of Legendre type in the sense of Rockafellar \cite[Chap. 26]{Rockafellar-book}. Let us now extract the following result from the classical theory of Legendre duality.

\begin{theorem}\label{theo-Legendre}
For each $1\leq r\leq\omega$, the Legendre transform of any concave or convex function $f$ of $C^r$ Legendre type is a convex function $f^\#$ or $f^*$ of $C^r$ Legendre type. 
Moreover, the following holds for $f$ concave:\footnote{For convex $f$, the same holds for $f^*$ except for the minus signs: $\grad f^*(y)=(\grad f)^{-1}(y)$ and $f^{**}=f$.} 
 \begin{enumerate}
  \item $\grad f:\inter(\dom(f))\to\inter(\dom(f^\#))$ is a $C^{r-1}$-diffeomorphism;
  \item\label{enumi:duality} for all $y\in\inter(\dom(f^\#))$,
 $$
  \grad f^\# (y) = (\grad f)^{-1}(-y) \text{ and  }
  f^\#(y)=z\cdot y+f(z) \text{ with }z=(\grad f)^{-1}(-y);
  $$
  \item $(f^{\#})^*=-f$.
 \end{enumerate}
\end{theorem}

\begin{proof}
This statement follows from the results in \cite[Chap.~26]{Rockafellar-book}, except for the formula for $f^\#(y)$ in (\ref{enumi:duality}). When $r=1$, this is exactly Theorem 26.5 there applied to the convex function $g=-f$. Indeed,  $f^\#=g^*$ and $\grad f = I\circ\grad g$ with $I(y)=-y$. 
In particular, $\grad g^* = (\grad g)^{-1}$, i.e., $\grad f^\# = (I\circ\grad f)^{-1} =(\grad f)^{-1}\circ I$, proving the first formula in claim (\ref{enumi:duality}).

Now, $\grad f$ is a $C^{r-1}$ map. From the same theorem, $\grad f:\dom(f)\to\dom(f^\#)$ is a homeomorphism. It is  a $C^{r-1}$-diffeomorphism, using, if $r\geq2$, that the Hessian of $f$ is definite. The formula for $\grad f^\#$ ensures that this gradient is also $C^{r-1}$, thus $f^\#$ is $C^r$.

To conclude, let $y\in\inter(\dom(f^\#))$. Note that  $ z :=  (\grad f)^{-1}(-y)\in\inter(\dom(f))$ satisfies $\grad_{ z}\,( y\cdot z+f( z))=0$. Since $f$ is strictly concave on $\inter(\dom(f))$ and concave everywhere, $ z$ must be the unique maximizer on $\dom(f)$, proving the second half of (\ref{enumi:duality}).                                                                                                          
\end{proof}


\subsection{Application to dynamical systems}
Before exploiting Legendre duality further, let us discuss how the dynamical systems on which the linear Thermodynamical formalism is well-understood fit into our framework. We start with a convenient definition.

\begin{definition}\label{defi-regular}
For $1\leq r\leq\omega$, a continuous dynamical system with potentials $(T,\vpot)$ is \emph{$C^r$ Legendre} when:
\begin{enumerate}
\item\label{item-interior} the rotation set $\rs(\vpot)$ has non-empty interior in $\mathbb{R}^d$,
\item the topological entropy is finite: $h_\top(T)<\infty$;
\item the finite-dimensional entropy function $\hf:\mathbb{R}^d\to\mathbb{R}\cup\{-\infty\}$ is concave of $C^r$ Legendre type. 
\end{enumerate}

If moreover, for every $ y\in\RR^d$, there is exactly one linear equilibrium measure $\nu_{ y}$ for $T$ and the potential $ y\cdot\vpot$, then we say that $(T,\vpot)$  is $C^r$ Legendre \emph{with unique linear equilibrium measures}  $(\nu_{ y})_{ y\in\RR^d}$.
\end{definition}


The above classical theory of Legendre duality {\JB applied} to such systems leads to the (finite-dimensional linear) \emph{pressure function}:
 $$
    \Pf( y) := \sup_{\mu\in\Prob(T)} h(T,\mu)+ \mu( y\cdot \vpot).
 $$
It is the Legendre transform of the concave finite-dimensional entropy function~$\hf$:
 $$
    \Pf( y) = \hf^\#(y) := \sup_{ z\in\rs(\vpot)} \hf( z) + \langle  y; z\rangle.
 $$
 
In particular, if $(T,\vpot)$ is $C^r$ Legendre, then by applying Theorem \ref{theo-Legendre} we obtain that the pressure is a $C^r$ function.

In Definition \ref{defi-regular}, we took entropy as primary object, and then defined pressure by Legendre duality. However, it has been customary to discuss primarily the regularity of pressure -- using Legendre duality, both points of view can be unified as follows.

\begin{proposition}\label{prop:nlem0}
If $(T,\vpot)$ is a continuous system with potentials satisfying, for some $1\leq r\leq\omega$,
\begin{itemize}
 \item the rotation set $\rho(\vpot)$ has nonempty interior {\JB in $\RR^d$};
 \item the entropy function $h(T,\cdot)$ is upper semicontinuous and bounded over $\Prob(T)$;
 \item the finite-dimensional pressure {\JB function} $\Pf$ is finite over $\RR^d$, $C^r$ smooth,  strictly convex and,  when $r\ge 2$, with everywhere positive definite Hessian,
\end{itemize}
then $(T,\vpot)$ is a $C^r$ Legendre system.
\end{proposition}

\begin{proof}
Since the Kolmogorov-Sinai entropy $h:\Prob\to\RR\cup\{-\infty\}$ is upper semicontinuous, $\Prob$ compact,  and $\vpot$ is continuous,  $\hf:\RR^d\to\RR\cup\{-\infty\}$ is upper semicontinuous. This function is also finite on its nonempty domain $\dom(\hf)=\rs(\pot)$ and concave. Therefore, by Theorem~\ref{theo-duality-CO}, the lower semicontinuous convex function $-\hf$ satisfies: $-\hf=((-\hf)^*)^*=\Pf^*$. By assumption $\Pf$ is a convex $C^r$ Legendre function. Applying now Theorem~\ref{theo-Legendre}, we get that  $\Pf^*=-\hf$ is a convex $C^r$ Legendre function, i.e., $\hf$ is concave $C^r$ Legendre.
\end{proof}

It is now easy to check that many classical systems satisfy the thermodynamical formalism with $C^\omega$ regularity. In many cases, the one point that needs checking is that the rotation set has non-empty interior (see Section \ref{subsec:potts} for an example where it does not).

Recall that a function $\pot$ is \emph{cohomologous to a constant} $c$ if there is a continous function $u$ such that $\pot=c+u-u\circ T$. 

\begin{corollary}
Let $T$ be a mixing subshift of finite type or an Anosov diffeomorphism. Let $\vpot$ be a finite family of H\"older-continuous potentials $\pot_1,\dots,\pot_d:X\to\RR$.
Assume the following independence condition: for all $\alpha_1,\dots,\alpha_d$ not all zero, $\sum_{i=1}^d\alpha_i\pot_i$ is not cohomologous to a constant.

 Then $(T,\vpot)$ is a $C^\omega$ Legendre system with unique linear equilibrium measures.
\end{corollary}

\begin{remark}
Livsi\v{c} theorem applies to such systems: a function is cohomologous to a constant if and only if on each periodic orbit, the average of the function is equal to that constant. The independence condition above is therefore equivalent to the existence of $d+1$ periodic orbits with corresponding atomic measures $\mu_0,\dots,\mu_d\in\Prob(T)$ such that $\mu_0(\vpot),\dots,\mu_d(\vpot)\in\RR^d$ are affinely independent.
\end{remark}

\begin{proof}[Proof of the corollary]
Both subshifts of finite type and Anosov diffeomorphisms are Smale systems satisfying the regularity condition (SS3) in \cite{Ruelle-book}  in the sense of \cite[7.1, 7.11]{Ruelle-book} and this will be enough for our purposes.

Since $T$ has finite topological entropy and is expansive, the Kolmogorov-Sinai entropy function is upper semicontinuous and bounded over $\Prob(T)$.

If the rotation set, a convex set, had empty interior, it would be contained in some  affine hyperplane, hence, there would be numbers $\alpha_0,\dots,\alpha_d$, not all zero, such that
 $$
    \forall\mu\in\Prob(T)\quad \mu\left(\sum_{i=1}^d \alpha_i \phi_i\right) =  \sum_{i=1}^d \alpha_i \mu(\phi_i) = \alpha_0.
  $$
By Livsi\v{c} theorem, this  implies that $\sum_{i=1}^d \alpha_i \phi_i$ is cohomologuous to the constant $\alpha_0$, contradicting the independence assumption.

Since $T$ is a topologically mixing Smale system,  its pressure function is real-analytic \cite[7.10]{Ruelle-book}. It has a semidefinite positive Hessian with kernel  generated by the potentials cohomologous to constants. Hence the finite-dimensional pressure function $\Pf$ has definite positive Hessian in all of $\RR^d$ under the independence assumption above. In particular, $\Pf$ is strictly convex.

Thus, the assumptions of Proposition~\ref{prop:nlem0} are satisfied so that $(T,\vpot)$ is a $C^r$ Legendre system.
 
Finally, for each $ y\in\RR^d$, $ y\cdot\vpot$ is H\"older-continuous, hence there exists a unique linear equilibrium measure $\nu_{ y}$.
\end{proof}

The next statement follows immediately from \cite[Corollary B, Theorems F \& G]{GKLM}, providing another family (intersecting the previous one) of dynamical systems to apply our framework to.
We shall say that a Banach space $\mathscr{X}$ of functions $X\to \mathbb{R}$ is a \emph{good Banach algebra of functions} when:
\begin{itemize}
\item $\mathscr{X}$ is stable by product and $\lVert fg\rVert \le \lVert f\rVert \lVert g\rVert$ for all $f,g\in\mathscr{X}$,
\item for every positive, bounded away from $0$ function $f\in\mathscr{X}$, $\log f$ is in $\mathscr{X}$,
\item the norm of $\mathscr{X}$ dominates the uniform norm (in particular the elements of $\mathscr{X}$ are bounded),
\item the composition operator $f\mapsto f\circ T$ is a continuous operator on $\mathscr{X}$,
\item for every equilibrium measure $\mu$ of a potential in $\mathscr{X}$ and every non-negative $f\in\mathscr{X}$, if $\int f\dd\mu = 0$ then $f=0$,
\item every continuous function can be uniformly approximated by elements of $\mathscr{X}$.
\end{itemize}
(These assumptions are numerous, but many Banach spaces satisfy them, such as H\"older spaces or BV space on the interval, see \cite{GKLM} for some discussions of these hypotheses.) We refer to \cite{GKLM} for the notions of $k$-to-$1$ map, simple dominant eigenvalue, and spectral gap appearing in the following statement.

\begin{theorem}\label{thm-classical-are-Legendre}
Assume that $T$ is $k$-to-$1$ and $\pot_1,\dots, \pot_d$ belong to some good Banach algebra of functions $\mathscr{X}$ and that for all $\alpha_1,\dots\alpha_d$ not all zero, $\sum_{i=1}^d \alpha_i\pot_i$ is not cohomologous to a constant. If for all $ y\in\mathbb{R}^d$ the transfer operator defined by  $\mathcal Lf (x) = \sum_{x'\in T^{-1}(x)} e^{ y\cdot\vpot(x')} f(x')$ acts with a simple dominant eigenvalue and a spectral gap on $\mathscr{X}$, then $(T,\vpot)$ is $C^\omega$ Legendre with unique linear equilibrium measures.
\end{theorem}

\subsection{Consequences of Legendre duality}

Now that we have seen that Theorem~\ref{theo-Legendre} applies to plenty of dynamical systems, let us note some of the consequences.

\begin{proposition}\label{prop:nlem}
If $(T,\vpot)$ is a $C^r$ Legendre system, then:
\begin{enumerate}
\item the finite-dimensional function $\hf$ is continuous on the rotation set $\rs(\vpot)$,
\item $\grad \hf$ realizes a $C^{r-1}$ diffeomorphism from the interior of $\rs(\vpot)$ onto $\mathbb{R}^d$ with inverse $ y\mapsto\grad \Pf(- y)$,
\item the linear pressure function $\Pf$ has domain $\mathbb{R}^d$ and is $C^r$, 
\item for all $ y\in\mathbb{R}^d$, $\grad \Pf( y) = z_{\mathrm{opt}}$ where $ z_{\mathrm{opt}}$ is the unique maximizer of $\hf( z)+\langle y; z\rangle$ over $\interior\rs(\vpot)$.
\end{enumerate}
If, additionally, $(T,\vpot)$ has unique equilibrium measures $(\nu_y)_{y\in\RR^d}$, then 
\begin{enumerate}[resume]
\item\label{eq-int-ent} for all $y\in\RR^d$, $\nu_{ y}(\vpot)=\nabla\Pf( y)\in\interior(\rs(\vpot))$
and
   $h(T,\nu_{ y})=\hf(\nu_{ y}(\vpot))$, 
\item\label{eq-eq-values} $\{\nu_{ y}(\vpot): y\in\RR^d\} = \interior(\rs(\vpot))$ and
\item\label{eq-rev} conversely, for all $ z\in\interior(\rs(\vpot))$, setting $y:=-\grad \hf(z)$,
    $\nu_{ y}(\vpot) =  z$ and $\nu_{ y}$ is the unique measure of maximum entropy in $\mathcal M( z)$.
\end{enumerate}
\end{proposition}

\begin{proof}
The function $\hf$ is upper-semicontinuous, and since it is concave and finite it must be continuous on its domain, which coincides with the rotation set.

By assumption, $\hf$ is a concave $C^r$ Legendre function. Hence Theorem~\ref{theo-Legendre} ensures that the pressure $\Pf=\hf^\#$ is $C^r$. Since $\hf$ is upper bounded as a continuous function with a compact domain, the domain of $\Pf(y)=\sup_{z\in\rs(\vpot)} \hf(z)+\langle y;z\rangle$ is the whole of $\RR^d$. 
The same theorem tells us that  $\grad \hf$ realizes a $C^r$ diffeomorphism  from the interior of $\rs(\vpot)$ to $\RR^d$, the interior of the domain of $\Pf$,  and that, for all $ y\in\dom(\Pf)$,
 $$
   \grad \Pf( y) = (\grad \hf)^{-1}(- y).
 $$
 We further note that $\Pf( y)=\langle  y\,; z_{\rm opt}\rangle + \hf(z_{\rm opt})$ with
 $
   z_{\mathrm{opt}} := (\grad \hf)^{-1}(- y)=\grad P( y).
 $

We now assume that $(T,\vpot)$ has unique equilibrium measures $(\nu_y)_{y\in\RR^d}$.
Let $ y\in\RR^d$. 

Observe that $\nu_y$ must maximize the entropy in $\mathcal{M}( z)$ where $ z =\nu_{ y}(\vpot)$, hence $h(T,\nu_{ y})=\hf( z)$. By definition the linear pressure is 
 $$
  \Pf( y)=  h(T,\nu_{ y})+\int  y\cdot \vpot\, d\nu_{ y} = \hf( z)+ \langle  y\cdot  z \rangle.
 $$
Therefore, in Proposition~\ref{prop:nlem}, one must have:
 $$
     z = \nu_{ y}(\vpot) = \grad\Pf( y) \text{ so }  y = -\grad\hf( z).
 $$
This proves items (\ref{eq-int-ent})~and~(\ref{eq-rev}).

Note that $\{\nu_{ y}(\vpot): y\in\RR^d\} = \grad\Pf(\RR^d)$, which is $\interior(\rs(\vpot))$, proving (\ref{eq-eq-values}). 
\end{proof}

\begin{figure}
\centering
\includegraphics[scale=1]{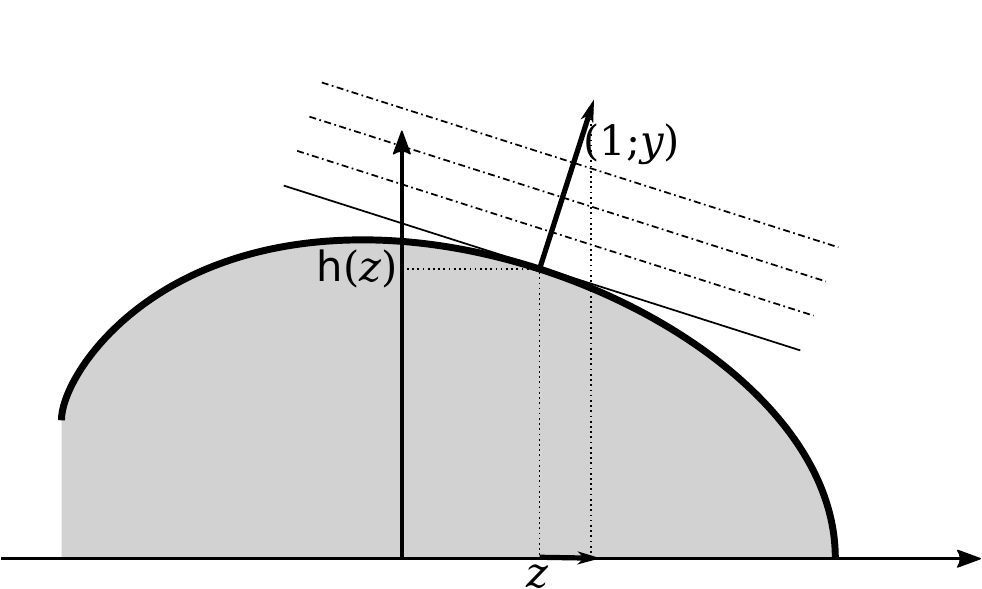}
\caption{An entropy-potential diagram $\mathcal{D}$ represented in {\JB the $d=1$ case} (first coordinate $z_0$ represented by the vertical axis): $\Pf( y)$ is obtained by sliding a line along the normal vector $(1; y)$ until it touches the hypograph of $\hf$, which happens above some $ z$ where $\grad \hf( z) = - y$. At this point $\grad \Pf( y)= z$: changing the direction $ y$ makes the touching line ``roll'' along the upper side of $\mathcal{D}$; this rolling combines the rotation of $y$ and a normal translation given by scalar product with $ z$. {\BK \tiny Changed "variation in the amount of sliding" by a hopefully clearer explanation.}}\label{fig:diagramme1Da}
\end{figure}

\subsection{Set of nonlinear equilibrium measures}

We now identify the fully nonlinear equilibrium measures, that is, the elements of $\EM(T,F,\vpot)$ (or just $\EM$) from Definition~\ref{def-fully-nonlin}. We define the set of \emph{$(G,\vpot)$-equilibrium values} to be
 $$
     \mathscr{V}:=\{\mu(\vpot):\mu\in\EM\}.
 $$ 

For $ z\in \rs(\vpot)$, recall the notations $\mathcal{M}( z)$ 
and $\hf( z)$ 
from Definitions \ref{defi-entropy}~and~\ref{defi-regular}.
We start with Theorem \ref{thm-equilibria}, in a version generalized to fully nonlinear pressures (see Definition~\ref{def-fully-nonlin}).
We recall that $G$ is defined on some open set $V\subset \R\times \R^{d}$ and in the following $\partial_{i}G$ stands for $\partial G/\partial z_{i}$, $i=0,1,\ldots, d$.

\begin{theorem}\label{thm:equilibria-G}
Let $(T, \pot)$ be a {\JB $C^r$ Legendre system} for some $1\leq r\leq\omega$ and let $\Pi^G$ be a fully nonlinear pressure defined by an admissible $C^r$ function $G$. 

Then the set $\EM$ of $(G,\vpot)$-equilibrium measures is a nonempty and compact set of linear equilibrium measures.
More precisely, 
 \begin{enumerate}
   \item\label{enumi:equilibria-G-i} $\mathscr V = \{ z\in\interior(\rs(\vpot)):G(\hf( z); z)\text{ maximal }\}$ is a nonempty compact set on which
     \begin{equation}\label{eq-gradg}
   0 = \grad G + \partial_{0} G \cdot  \grad \hf \qquad \text{ where } \grad := \left(\partial_1,\dots,\partial_d\right).
 \end{equation}
   \item\label{enumi:equilibria-G-ii} $ \EM = \{\nu_{ y} \colon  y \in -\grad \hf(\mathscr{V})\}$.
 \end{enumerate}
\end{theorem}

\begin{proof}
We prove assertions (\ref{enumi:equilibria-G-i}) and (\ref{enumi:equilibria-G-ii}), the rest being immediate consequences.

Let us note that a measure $\mu\in\Prob(T)$ is a fully nonlinear equilibrium measure if and only if
 $$
     G(h(T,\mu);\mu(\vpot)) = \sup_{(z_0; z)\in\mathcal D} G(z_0; z)
      = \sup_{ z\in\rs(\vpot)} g( z) \text{ where }g( z):=G(\hf( z); z).
 $$
Indeed, the first equality follows from the definitions and the second one 
follows from the fact that $z_0\mapsto G(z_0; z)$ is increasing for each $ z\in\rs(\vpot)$. Since $g$ is continuous on the compact set $\rs(\vpot)$, it follows that $\mathscr V$ is itself compact. 

\begin{claim*}
Since $h$ is concave with $\lvert \grad\hf \rvert\to\infty$ at the boundary of $\rs(\vpot)$, we must have $\mathscr V\subset\inter(\rs(\vpot))$.
\end{claim*}

\begin{proof}[Proof of the claim]
Consider a point $z_0$ on the boundary of $\rs(\vpot)$, and let us prove that it cannot maximize $g$. Let $\vec u$ be any vector such that $z_0+\vec u\in\inter(\rs(\vpot))$ and consider the function defined on $[0,1]$ by $f(t) = \hf(z_0+t\vec u)$. By concavity its derivative has a limit, finite or infinite, as $t\to0$. For all {\JB small enough} $t>0$, we have $f'(t) = \langle\grad\hf(z_0+t\vec u),\vec u\rangle$. We know that $\lvert\grad \hf\rvert\to\infty$ at the boundary, but it could {\it a priori} be that $\grad \hf$ becomes orthogonal to $\vec u$ as $t\to0$; we now prove that this cannot be the case.

At each {\JB small enough} $t>0$, the tangent space $H_t$ to the upper boundary of $\mathcal{D}$ has $(1,-\grad\hf)$ as normal vector. As $t\to0$, $\lvert \grad\hf\rvert\to\infty$ so that any accumulation point $H_0$ of $H_t$ is vertical, of the form $\mathbb{R}\times L$ where $L$ is a hyperplane of $\mathbb{R}^d$ (normal to an accumulation point of the direction of $\grad\hf(z_0+t\vec u)$). Since $\mathcal{D}$ is contained in a half-space delimited by $H_0$, $L$ must be a supporting hyperplane of $\rs(\vpot)$ at $z_0$. Since $\vec u$ has been chosen pointing to the interior of $\rs(\vpot)$, the angle between $\vec u$ and $L$ is bounded away from $0$. It follows that for some constant $c>0$ and all $t>0$, $\langle\grad\hf(z_0+t\vec u),\vec u\rangle\ge c\lvert \grad\hf(z_0+t\vec u)\rvert \lvert \vec u\rvert\to\infty$.

We deduce that $f'(t)\to+\infty$ as $t\to0$. Since $\partial_0G>0$, it is bounded away from $0$ on the segment of endpoints $z_0$ and $z_0+\vec u$ and it follows that $g(z_0+t \vec u)-g(z_0)\gg t$ as $t\to0$. In particular there exists $t>0$ such that $g(z_0+t \vec u)>g(z_0)$.
\end{proof}

It follows that $\grad g = 0$ on $\mathscr V$. Now, 
  $$
      \grad g = \grad G + \frac{\partial G}{\partial z_0} \grad \hf 
  $$
and eq.~\eqref{eq-gradg} follows and assertion~(\ref{enumi:equilibria-G-i}) is established.

Let $\nu\in\EM$. The above remarks show that $\nu$ maximizes the entropy in $\mathcal M(z)$ where  $ z:=\nu(\vpot)$. By Proposition~\ref{prop:nlem}, this implies that $\nu=\nu_{ y}$ where $ y:=-\grad\hf( z)$, yielding the inclusion
 $$
    \EM \subset \{\nu_{ y}: y\in -\grad\hf(\mathscr V)\}.
 $$
To check the converse inclusion, let $ z\in\mathscr V$ and apply Proposition~\ref{prop:nlem}. Setting $y:=-\grad\hf( z)$ so $ z:=\grad\Pf( y)$, we get $G(h(T,\nu_{ y});\nu_{ y}(\vpot))=g( z)$ which is maximum since $ z\in\mathscr V$. Hence $\nu_{ y}\in\EM$. Assertion~(\ref{enumi:equilibria-G-ii}) is established.
\end{proof}

\begin{remark}
The value $\max_{\Prob(T)} \FNLP$ is a generalization of our previous definition of nonlinear pressure. Of course, one could decide to study the variational principle for full general $G$ without any restriction. Nevertheless we point out that:
\begin{enumerate}
\item Assumption $\inf\partial_0 G>0$ is crucial: a change of sign would modify the nature of the problem,
\item the case $G(z_0; z) = z_0+F( z)$ is of particular interest: in the classical variational principle, the term $h(T,\mu)$ comes from the summation over $(\eps,n)$-covers in the Gibbs measures (see Formula \eqref{eq:NLP}), and there is at the moment no candidate to replace this summation and define a \emph{topological} pressure in the case of a general $G$.
\end{enumerate}
\end{remark}

To state our next result, we recall that a \emph{subvariety} of an open set $W\subset\RR^d$ is a subset defined by finitely many functions $h_1,\dots,h_k\in C^r(W)$ as $\{x\in W: h_1(x)=\dots=h_k(x)\}$. If $r=\omega$, it is easy to see that any nontrivial subvariety has zero Lebesgue measure (see, e.g., \cite{Mityagin} for a simple proof).

The previous theorem implies the following, which in particular contains Theorem \ref{thm-finiteness}.

\begin{corollary}
Let $(T, \vpot)$ be a $C^\omega$ Legendre system and $G$ be a $C^\omega$ admissible function defined on an open set $V\subset \mathbb{R}^{1+d}$. Then the set $\mathscr{V}$ of $(G,\vpot)$-equilibrium values is a compact subset of an analytic sub-variety of $\mathbb{R}^d$.

In particular, it is a closed set with empty interior which is Lebesgue negligible. 
\end{corollary}

Since a proper analytic sub-variety of a compact line segment is finite:

\begin{corollary}\label{cor-finite-1D}
Let $(T,\pot)$ be a $C^\omega$ Legendre system and $G$ be a $C^\omega$ admissible with  $d=1$, then the set $\EM$ of equilibrium measures is finite.
\end{corollary}

In full generality, we have a \emph{generic} uniqueness:

\begin{proposition}\label{p-generic-unique}
Let $(T, \vpot)$ be a Legendre $C^r$ regular system for some $2\leq r\leq\omega$. There is a unique nonlinear equilibrium measure in both of the following settings:
\begin{enumerate}
 \item\label{enumi:generic1} For $G$ in some open and dense subset of $\{G\in C^r(V):\partial_0 G>0\}$ where $V$ is a given admissible open subset of $\RR\times\RR^d$;
 \item\label{enumi:generic2} For $G(z_0;z)=z_0+F(z)$ with $F$ in some open and dense subset of $C^r(U)$ where $U$ is a given open neighborhood of $\rs(\vpot)$ in $\RR^d$.
\end{enumerate}
\end{proposition}

Claim (\ref{enumi:generic2}) above means that, for a generic nonlinearity $F$, there is a unique nonlinear equilibrium measure. It is not implied by the fully nonlinear case (\ref{enumi:generic1}) since the corresponding set of $G$s has empty interior. It would be interesting to determine conditions on a fixed non-linearity $F$ or $G$ under which a generic $\vpot$ leads to a unique equilibrium measure

\medbreak

In higher dimension $d\ge 2$, we do not know any example with $C^\omega$ regularity where finiteness does not hold. Beyond the real analytic case, even finiteness fails to hold for arbitrary nonlinearity:

\begin{proposition}\label{prop-flexibility}
Let $(T, \vpot)$ be a {\JB $C^r$ Legendre system} for some $2\leq r\leq\infty$.
For all compact $E\subset\interior \rs(\vpot)$, there exists a $C^r$ nonlinearity $F$ such that the set of equilibrium values $\mathscr{V}$ equals $E$. In particular the set of equilibrium measures can be infinite, even uncountable.
\end{proposition}

Before proving these two propositions, we recall some well-known facts about Morse functions. Given any open subset $U\subset\RR^d$,  a function $g\in C^r(U)$ with $2\leq r\leq\omega$ is Morse on $K\subset U$  if no critical point in $K$ is degenerate and it is nonresonant if it takes distinct values at each of its critical points in $K$ \cite[Def. 1.1.7 and 1.2.11]{Nicolaescu}. In particular, it has at most one maximizer on $K$. Finally, the set of nonresonant Morse $C^r$ functions on a compact set is open and dense (see the proofs in \cite[Sect. 1.2]{Nicolaescu}). This is to be understood with respect to the classical uniform topologies on $C^r(U)$ with finite $r$, or the limit topology for $C^\infty(U)$, or the more complicated standard topology of $C^\omega(U)$ (see, e.g., \cite[p. 53]{Krantz-Parks}).

\begin{proof}[Proof of Proposition \ref{p-generic-unique}]
We prove Claim (\ref{enumi:generic1}). The proof of Claim (\ref{enumi:generic2}) is entirely similar.
Note that it is enough to prove the claim under the auxiliary assumptions $\partial_0 G>1/C$ and $\lvert \grad G\rvert < C$ for $C>0$ arbitrary. 

First, note that $0=\grad g$ implies that $\lvert\grad h\rvert \leq C^2$. Hence, it is enough to ensure that $G$ is nonresonant Morse on the compact subset:
 $$
   K:=\{z\in V: \lvert\grad\hf\rvert\leq C^2\}.
 $$
Second, observe that $G\mapsto g$ is continuous from $C^r(V)\to C^r(\interior\rs(\vpot))$.
Therefore the set $\mathcal G$ of $G\in C^r(V)$ such that $g$  is nonresonant and Morse  on $K$ is open.

Third, given any $g\in C^r(V)$, the map $k\mapsto g+k$ is a self-homeomorphism of $C^r(V)$. Therefore there are arbitrarily small $k\in C^r(V)$ such that $g+k$ is nonresonant and Morse  on $K$. Considering $\tilde G(z_0,\dots,z_d):=G(z_0,\dots,z_d)+k(z_1,\dots,z_d)$ shows that $\mathcal G$ is dense in $C^r(V)$. 
\end{proof}

\begin{proof}[Proof of Proposition \ref{prop-flexibility}]
Let $f:\mathbb{R}^d\to [0,\infty)$ be a $C^\infty$ function such that $E=\{ z\in\mathbb{R}^d \mid f( z) = 0\}$ (such a function can be constructed as a convergent sum of functions that are each positive on one open balls, with the union of the balls equal to the complement of $E$). Let $F$ be $-1$ outside $\interior\rs(\vpot)$, coincide with $-f-\hf$ on a compact subset of $\interior\rs(\vpot)$ containing $E$ in its interior, and be lesser than $-\hf$ in between; such a function exists since $E$ does not approach the boundary of the rotation set.
Then maximizing $\hf( z)+F( z)$ is the same as minimizing $f( z)$, and is achieved precisely on $E$.
\end{proof}

{\JB
Since $\mathscr{Y}= -\grad h(\mathscr{V})$ where $-\grad h:\RR^d\to\inter\rs(\vpot)$ is a diffeomorphism, Proposition~\ref{prop-flexibility} gives:

\begin{corollary}\label{cor-flex-Y}
Let $(T, \vpot)$ be a Legendre $C^r$ regular system for some $2\leq r\leq\infty$.
For all compact sets $E\subset\RR^d$, there exists a $C^r$ nonlinearity $F$ such that the set  $\mathscr{Y}$ from Theorem~\ref{thm-equilibria} equals $E$. In particular the set of equilibrium measures can be infinite, even uncountable.
\end{corollary}
}

\section{Examples of phase transitions}\label{sec-examples}

This section is devoted to the application of the framework developed above to a few families of systems whose energy depends on a real multiplicative parameter (i.e., an inverse temperature) and exhibiting various behaviors when this parameter is modified: changes in the number of equilibrium measures, piecewise analytic behavior with or without an affine piece. Most examples belong to the non-linear thermodynamical formalism, but even in the linear case we provide new insight thanks to the entropy-potential diagram $\mathcal{D}$, see Theorem \ref{thm-fpt-linear}.

\subsection{The Curie-Weiss Model - Symmetric case}\label{sec-Curie-Weiss}

The \emph{Curie-Weiss energy} for a potential $\pot$ is given by a quadratic nonlinearity, i.e., $\En(\mu) = \beta\En_1(\mu) = \frac{1}{2} \beta \mu(\pot)^2$ where $\beta$ is a parameter called the \emph{inverse of temperature}. For this specific case, we shall first use our general machinery above to recover an example treated in \cite{Leplaideur-Watbled}, then provide a second example exhibiting a ``metastable'' phase transition.

We consider here the left shift $T$ on $X:=\{a,b\}^\N$, endowed for example with the distance 
 $$
    d(x,y)=2^{-\inf\{i \mid x_i\neq y_i\}} \qquad \text{where } x=(x_i)_{i\in\mathbb{N}}, y=(y_i)_{i\in\mathbb{N}},
 $$
with the potential  $\pot:X\to\RR$ defined by 
 $$
    \pot(x)=\begin{cases}
      -1 &\text{if }x_0 = a \\
      1 &\text{if }x_0 = b
    \end{cases}
 $$
and the Curie-Weiss nonlinearity $F(z) = \beta F_1(z):=\frac{1}{2}\beta z^{2}$, with $\beta\ge 0$.

For any given $z\in[-1,1]$, we consider the invariant measures $\mu\in\mathcal{M}(z)$, i.e., such that $\mu([b])-\mu([a])=z$ where $[i]$ is the cylinder of words starting with the letter $i$. Since these two cylinders form a partition of $X$, this equation rewrites as $\mu([a]) =\frac{1-z}{2}$ (and therefore $\mu([b])=\frac{1+z}{2}$). 
\jb{}
Among invariant measures in $\mathcal{M}(z)$, the one of maximal entropy is the Bernoulli measure with weights $(\frac{1-z}2,\frac{1+z}2)$, whose entropy is well-known:
 $$
    \hf(z) = -\frac{1-z}{2}\log\frac{1-z}{2} -\frac{1+z}{2}\log \frac{1+z}{2}
 $$
We thus are left with maximizing, given $\beta\ge 0$,
 $$
   P_\beta(z):= \hf(z)+ \beta F(z) =-\frac{1-z}2\log\frac{1-z}2-\frac{1+z}2\log\frac{1+z}2 + \frac12 \beta z^{2}.
   $$ 
A simple computation shows that there are two cases (see Figure \ref{fig:CurieWeiss}):
\begin{enumerate}
\item For $0\le \beta\le 1$, $0$ is the unique critical point of $\P_\beta$ and is indeed a maximum. Thus, $\mathscr{V}=\{0\}$, there is a unique equilibrium state which is the Bernoulli measure of weights $(\frac12,\frac12)$, and the nonlinear topological pressure is $\NLP^{\beta\En_1}_\top(T) = \log 2$.
\item For $\beta>1$, there are three distinct critical points $\{-z_\beta,0,z_\beta\}$ among which $0$ is a local minimum and $-z_B<z_B$ are two global maxima.
Hence, $\mathscr{V}=\{-z_\beta,z_\beta\}$ and there are two equilibrium measures, which are ``symmetrical'' Bernoulli measures, one with $\mu([a])=\frac{1-z_\beta}{2}$ the other with $\mu([b])=\frac{1-z_\beta}{2}$.
\end{enumerate}
We have recovered the result of  \cite{Leplaideur-Watbled} that the nonlinear equilibrium measure is unique for $0\leq \beta\leq 1$ but that there are two of them for $\beta>1$, in line with the physical model.

\begin{figure}
\centering
\includegraphics[scale=1]{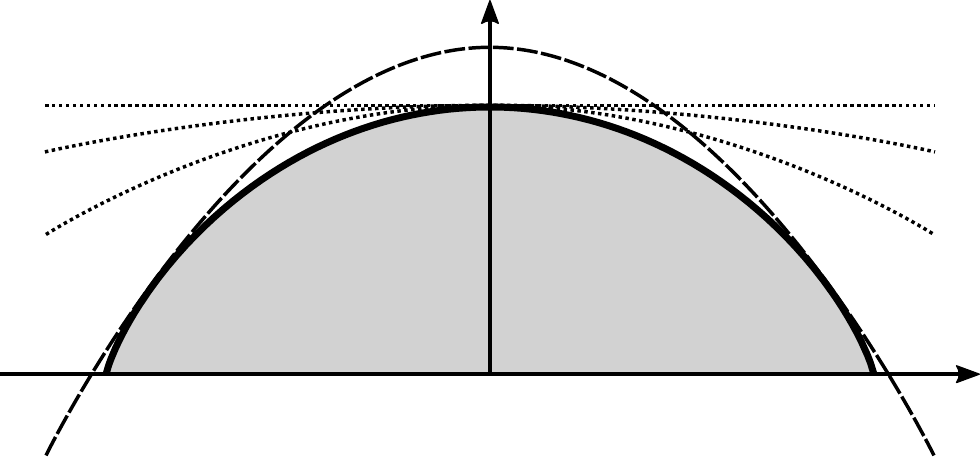}
\caption{The symmetric Curie-Weiss example: graph of $\hf$ (solid line), highest translates of the graph of $-\beta F$ touching it (dotted line: $\beta<1$; dashed line: $\beta>1$).}\label{fig:CurieWeiss}
\end{figure}

Note that any $C^2$ Legendre system $(T,\pot)$ with an entropy-potential diagram that is symmetric with respect to the vertical axis will provide a similar example. Indeed the symmetry ensures that for all $\beta$, $0$ is a critical point; and as long as $\beta<\hf''(0)$, the graph of $\hf$ being more concave at $0$ than the graph of $-\beta F$, $0$ will be a local maximum. It will then be a global maximum at least when $\beta$ is close enough to $0$. 
For $\beta>\hf''(0)$, $0$ will be a local minimum and one will get (at least) two non-zero symmetric equilibrium values.

\subsection{An asymmetric Curie-Weiss model}\label{sec-metastable}

Consider now the space of three-letter words $X=\{a,b,c\}^{\mathbb{N}}$ 
and let $T$ be the left shift on $X$. We will again consider the Curie-Weiss nonlinearities $F(z) = \beta F_1(z) = \beta\frac{z^2}{2}$ where $\beta\in[0,+\infty)$ is the inverse of the temperature, but with a potential exhibiting a specific asymmetry:
\[\pot(x) = \begin{cases}
-2 &\text{when }x_0=a \text{ or }x_0=b,\\
3 &\text{when }x_0=c
\end{cases}\]
Here $\rho(\pot)=[-2,3]$ and a measure maximizing entropy under the constraint $\mu(\pot)=z$ must, as above, be a Bernoulli measure. If we write $(p,q,1-(p+q))$ for its weights, the constraint translates as
\begin{equation}
p+q = \frac{3-z}{5}.
\end{equation}
Given this constraint, it is easily checked that entropy is maximized when $p=q$. Setting $p(z) = (3-z)/10$, we get that the measure in $\mathcal{M}(z)$ maximizing entropy is the Bernoulli measure with weights $(p(z),p(z),1-2p(z))$ and we obtain
\begin{align*}
    \hf(z) &= -2p(z)\log p(z)-(1-2p(z))\log(1-2p(z)) \\
         &= \frac{z-3}{5} \log\frac{3-z}{10}-\frac{2+z}{5}\log \frac{2+z}{5}.
\end{align*}
We are left with maximizing 
\[P_\beta(z) := \hf(z)+\beta F_1(z) = \frac{z-3}{5} \log\frac{3-z}{10}-\frac{2+z}{5}\log \frac{2+z}{5} +\frac12\beta z^2\]
for $z\in[-2,3]$. The critical points of $P_\beta$ are given by the intersections of the graph of $\hf'$ with the line $\ell_\beta=\{(z_0,z)\mid z_0=-\beta z\}$. We have
\begin{align*}
    \hf'(z) &= \frac{1}{5} \log\Big(\frac{3-z}{4+2z} \Big) \\
    \hf''(z) &= -\frac{1}{(2+z)(3-z)} = \frac15\Big(\frac{1}{z-3} -\frac{1}{z+2} \Big) \\
    \hf'''(z) &= \frac15\bigg(\frac{1}{(z+2)^2} -\frac{1}{(z-3)^2} \bigg)
\end{align*}
so that $\hf'$ is strictly decreasing, from $+\infty$ when $z\to-2$ to $-\infty$ when $z\to 3$; it has a single inflection point at $z=\frac12$, is convex on $(-2,\frac12]$ and concave on $[\frac12,3)$ (see its graph in Figure \ref{fig:hprime}). 


It follows that for $\beta\ge 0$ small enough, $P_\beta$ has only one critical point, which must be a maximum; in this regime, there is only one equilibrium state, with equilibrium value $z<0$, and the pressure varies analytically.

\begin{figure}
\centering
\includegraphics[width=.8\linewidth]{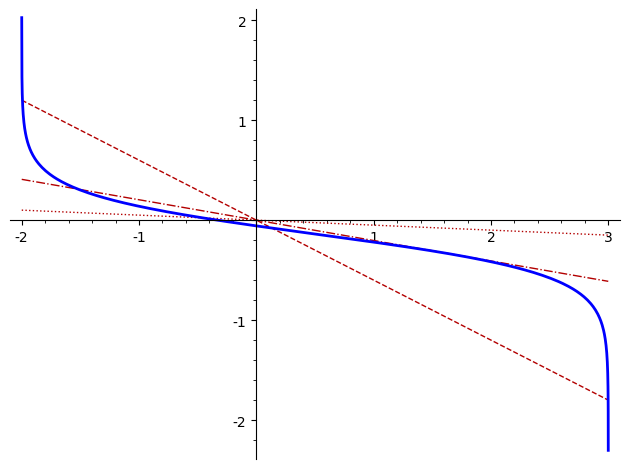}
\caption{The graph of $\hf'$ and $\ell_\beta$ for three values of $\beta$: $\beta<\beta_1$ (dotted), $\beta=\beta_1$ (dot dash), $\beta>\beta_1$ (dashed).}\label{fig:hprime}
\end{figure}

Increasing $\beta$, at some value $\beta_1$ the line $\ell_\beta$ touches the graph of $\hf'$ on the right, and a second critical point appears.
However, at this moment there is still only one equilibrium measure: $P_\beta$ is unimodal, decreasing around the second critical point. Increasing $\beta$ any further makes $P_\beta$ bimodal, with three critical points: one local minimum located between two local maximums $z_1(\beta)<z_2(\beta)$.

At first, $z_1(\beta)$ is the unique global maximum, but it ultimately gets surpassed by $P_\beta(z_2(\beta))$, precisely at the inverse temperature $\beta_0$ when the vertical translate of the graph of $-\frac{\beta}{2} z^2$ touching the graph of $\hf$ does so at two points. The choice of $\pot$ has been made to ensure this happens, by giving the entropy-potential diagram a larger overhang to the right than to the left (see Figure \ref{fig:metastable}): as $\beta\to\infty$, the highest translate of the graph of $-\beta F$ that touches the graph of $\hf$ converges to the two vertical lines of equations $(z=3)$ and $(z=-3)$. 
The latter of these vertical lines is far from the entropy-potential diagram since $\rs(\pot)=[-2,3]$, and for large enough $\beta$ the unique global maximum of $P_\beta$ must be attained at $z_2(\beta)\to 3$.

\begin{figure}
\centering
\includegraphics[scale=1]{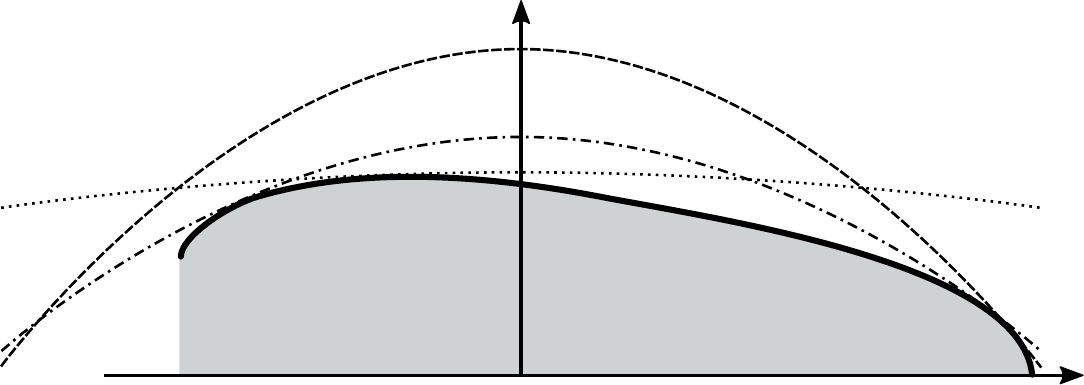}
\caption{A metastable phase transition: graph of $\hf$ (solid line, graph modified for readability), highest translates of the graph of $-\beta F_1$ touching it (dotted: $\beta<\beta_0$; dashed: $\beta>\beta_0$; dot-dashed: $\beta=\beta_0$).}\label{fig:metastable}
\end{figure}

Again the pressure is analytic for $\beta>\beta_0$, but we have a phase transition at $\beta_0$: the pressure is $\beta\mapsto \max(P_\beta(z_1(\beta)),P_\beta(z_2(\beta)))$ and cannot be analytical at the point where the arguments of the max cross each other. Observe that the value $\beta_1$ ($<\beta_0$) does not correspond to a phase transition: pressure is analytic in the vicinity of $\beta_1$.

This example motivates the following definition.
\begin{definition}
A system $(T,\En_1)$ is said to exhibit a \emph{metastable phase transition} at inverse temperature $\beta_0>0$ when there are two curves of invariant probability measures $(\mu_\beta)$, $(\nu_\beta)$ defined on a neighborhood $I$ of $\beta_0$ with $\beta\mapsto \NLP^{\beta\En_1}(\mu_\beta)$ and $\beta \mapsto \NLP^{\beta\En_1}(\nu_\beta)$ both $C^\omega$, such that:
\begin{enumerate}
\item for all $\beta\in I$, $(\mu_\beta)$, $(\nu_\beta)$ are local maximums of 
$\NLP^{\beta\En}$,
\item for $\beta<\beta_0$, $\mu_\beta$ is an equilibrium measure of $\beta\En$ but $\nu_\beta$ is not, and for $\beta>\beta_0$, $\nu_\beta$ is an equilibrium measure but $\mu_\beta$ is not.
\end{enumerate}
\end{definition}
Observe that the pressure function $\beta\mapsto \NLP^{\beta\En_1}$ is not analytic at $\beta_0$, for otherwise $\NLP^{\beta\En_1}(\mu_\beta)$ and $\NLP^{\beta\En_1}(\nu_\beta)$ would have to coincide and both $\nu_\beta$ and $\mu_\beta$ would be equilibrium measures throughout $I$.

{\JB
The ``metastable'' terminology is suggested by the analogy with the physical phenomenon of the same name. A simple example of it is that of water remaining liquid  below the freezing point in some circumstances. This is modeled by the liquid state (described by $\mu_\beta$) admitting a continuation to $\beta>\beta_0$ as a local maximum and the global maximal, the solid state (described by $\nu_\beta$), being too far from $\mu_\beta$ to allow the water to easily reorganize itself from one state to the other.
}

What we have proven can be summarized as follows.
\begin{maintheorem}\label{th-metastable}
There exists a locally constant potential $\pot$ on a full shift $X$ such that the Curie-Weiss energy $\En_1(\mu) = \frac12 \mu(\pot)^2$ exhibits a metastable phase transition.
\end{maintheorem}
This gives another concrete example of multiple nonlinear equilibrium measures in a context where the linear thermodynamical formalism is long known to be flawless (analytic pressure, etc.)

\subsection{The mean-field Potts model}\label{subsec:potts}

The mean-field Potts model is given by the full shift $(X,T)$ over a finite alphabet $\{\theta_{1},\ldots, \theta_{n}\}^{\N}$ or $\{\theta_{1},\ldots, \theta_{n}\}^{\Z}$, with $n\geq3$.
The potential is $\vec\pot:=(\BBone_{\theta_{1}},\ldots, \BBone_{\theta_{n}})$ and the nonlinearity $F(z)=\beta F_1(z)=\frac\beta2 \lvert \vec z\rvert^2$ where $\lvert\cdot\rvert$ is the usual Euclidean norm. The energy is thus given by 
$$\En(\mu = \beta\En_1(\mu)=\frac\beta2 \Big\lvert\int \vec\pot \dd\mu\Big\rvert^2 = \frac\beta2 \sum_i \mu([\theta_i])^2$$
where, as above, $[\theta_i]$ is a cylinder, the set of words having the letter $\theta_i$ in zeroth position.

The framework developed above seems not to apply since the potentials are not linearly independent up to (coboundaries and) constants: $\sum_i \BBone_{\theta_{i}} \equiv 1$, and the rotation set has empty interior. Let us take this as an opportunity to explain how this hypothesis is easily recovered: one simply extract a maximal independent subfamily of potentials, here $\vec\pot_\circ = (\BBone_{\theta_{1}},\ldots, \BBone_{\theta_{n-1}})$, and adjusts the nonlinearity to ensure $F_\circ(\mu(\vec\pot_\circ))=F(\mu(\vec\pot))$ for all $\mu\in\Prob(T)$, here      
 $$
   F_\circ(z_1,\dots,z_{n-1}) = \frac\beta2 \Big(z_1^2+\dots+z_{n-1}^2+\big(1-\sum_{i<n} z_i\big)^2\Big).
 $$
It is always possible to construct such an $F_\circ$, since by maximality each the potentials that are present in $\vec \pot$ can be expressed as linear combination of the potentials in $\vec\pot_\circ$ up to a coboundary and a constant, and a coboundary $g-g\circ T$ can be neglected since $\mu(g-g\circ T)=0$ for all invariant measures $\mu$.

Now $(T,\vec\pot_\circ)$ is {\JB $C^\omega$ Legendre} and we can apply Theorems \ref{thm-existence-NL} and \ref{thm-equilibria} (recall that moreover $(T,\vec\pot_\circ)$ has unique linear equilibrium measures, hence each $z\in\mathscr{V}$ yields a unique nonlinear equilibrium measure), and these results translate to the original system $(T,\vec\pot)$ with the nonlinearity $F$: accumulation points of Gibbs ensembles are convex combinations of the nonlinear equilibrium measures, each of which coincides with a linear equilibrium measure for some linear combination of the $(\pot_i)$; however, due to the lack of independence, several different linear combinations lead to the same equilibrium state.

In the specific case of the mean-field Potts model one can work out the equilibrium measures by (nontrivial) direct computations.
Given a vector $z:=(z_{1},\ldots, z_{n})$ in the rotation set 
 $$
    \rs(\vec\pot) :=\Big\{\int\vec\pot\,d\mu,\ \mu\in\Prob(T)\Big\} = \Big\{(z_1,\dots z_n) \in [0,1]^n \colon \sum_i z_i = 1\Big\},
 $$
the maximal entropy among invariant measures $\mu$ satisfying $\mu(\vec\pot)=z$ is  $\hf(z)=-\sum_{i}z_{i}\log z_{i}$ with the convention $0\log 0=0$. It is achieved by a unique measure, the Bernoulli measure giving each cylinder $[\theta_i]$ the mass $z_i$.

For $\be\ge 0$, the nonlinear pressure is 
$$\NLP_\top^{\beta\En_1}=\max_{\vec z}-\sum_{i}z_{i}\log z_{i}+\frac{\be}{2}\sum_{i}z_{i}^{2}.$$

We now summarize results from \cite{EllisWang90}.
For $0 \le \be<\be_{c}:=\disp 2\frac{n-1}{n-2}\log(n-1)$, $\NLP_{\top}^{\beta\En_1}$ is reached for $z=(\frac1n,\dots,\frac1n)$. 
The value is $\frac\be{2n}+\log n$ and is achieved by a unique measure.

For $\be>\be_{c}$,  $\NLP_{\top}^{{\BK \beta}\En_1}$ is given by an implicit equation. It is realized by $z$ equal to any permutation of $\widetilde{z}$ defined by 
$$\widetilde{z}_{1}=\frac{1+(n-1)s}n,\ \widetilde{z}_{i}=\frac{1-s}n,\ 2\le i\le n$$
where $s$ is the biggest solution for 
$$s=\frac{1-e^{-\be s}}{1+(n-1)e^{-\be s}}.$$
Each permutation of $\widetilde{z}$ gives a distinct equilibrium measure. Thus we get exactly $n$ equilibrium measures.

For $\be=\be_{c}$, the maximal value is simultaneously realized by ($\frac1n,\ldots ,\frac1n)$ and by the $n$ distinct permutations of $\tilde z$. Thus we get exactly $n+1$ equilibrium measures. In this case, the convergence of Gibbs measures to a convex combination of these equilibrium measures was previously shown in \cite{Leplaideur-Watbled-2}.

\subsection{Freezing phase transitions}\label{sec:transition}

Let us explain how the entropy-potential diagram can be used to visualize ``freezing phase transitions'', i.e., situations where for some $\beta_0$, the set of equilibrium measures of the energy $\beta\En_1$ is constant for $\beta>\beta_0$. These measures are called the \emph{ground states}. The physical interpretation is that once the temperature goes below some positive value $1/\beta_0$, the system freezes in  a macroscopic state corresponding to zero temperature, described by (one of) the ground states. In the linear thermodynamical formalism, the first freezing phase transition was exhibited by Hofbauer \cite{Hofbauer}, motivated by giving examples with multiple equilibrium states (this is sometimes achieved at $\beta=\beta_0$). Concretely, the typical examples are for the shift $T$ on $X = \{a,b\}^\mathbb{N}$ or $X = \{a,b\}^\mathbb{Z}$ with potentials
\[\pot(x) = -\frac{1}{k(x)^\alpha}, \qquad k(x) := \min\{\lvert k\rvert \colon x_k\neq a\}\]
with $\alpha\in(0,1]$, and the freezing equilibrium measure is $\mu_0=\delta_{aaa\dots}$. It has more recently been shown by Bruin and Leplaideur \cite{Bruin-Leplaideur1, Bruin-Leplaideur2} that one can produce in a similar way a freezing phase transition with more interesting ground states, supported on some uniquely ergodic, zero-entropy compact subsets of $X$ such as given by the Thue-Morse or the Fibonacci substitutions.

Let us interpret in the entropy-potential diagram $\mathcal{D}$ such a freezing phase transition, with potential $\pot$ being maximized by some invariant measure $\mu_0$, say with $\mu_0(\pot)=0$ for normalization. By definition, for $\beta\ge\beta_0$ the pressure is affine and achieved at $\mu_0$, meaning that all lines of slope $<-\beta_0$ touching $\mathcal{D}$ do it at the same point (see Figure \ref{fig:freezing}).

\begin{figure}
\centering
\includegraphics[width=.48\linewidth]{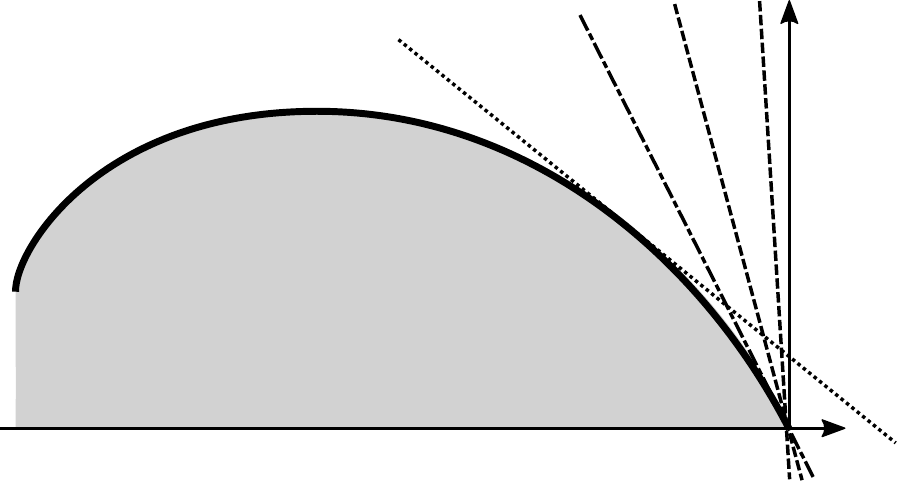}
\includegraphics[width=.48\linewidth]{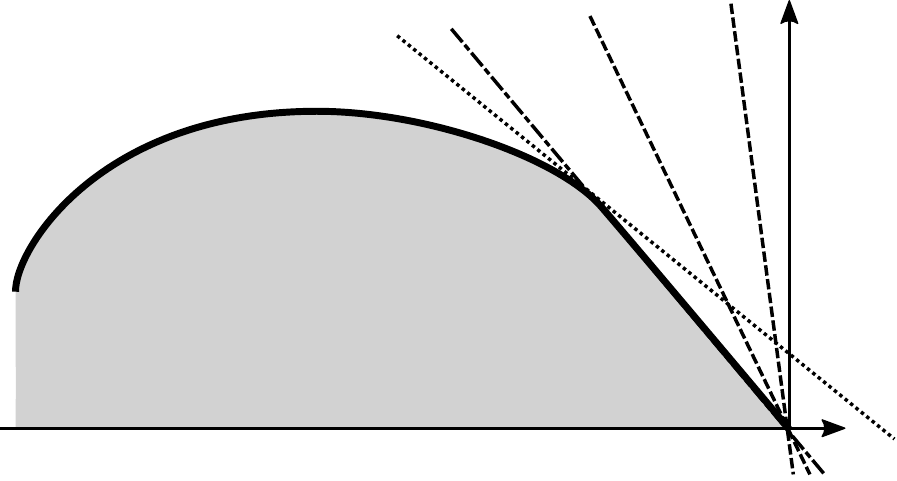}
\caption{Freezing phase transitions in the linear thermodynamical formalism: for $\beta>\beta_0$, all support lines are concurrent, and $\mathcal{D}$ must exhibit an acute corner at its right end. Left: $\hf$ is strictly concave, there might be a unique equilibrium measure throughout (case $\alpha=1$ in Hofbauer's example). Right: $\hf$ has a flat part, at $\beta_0$ there are (at least) two ergodic equilibrium measure, one at each end of the flat edge (case $\alpha<1$ in Hofbauer's example).}\label{fig:freezing}
\end{figure}

This observation immediately implies a characterization of (linear) freezing phase transition by a linear inequality between the entropy and the integral of the potential.

\begin{proposition}\label{prop-fpt}
Let $T:X\to X$ be a {\BK measurable map}, $\pot:X\to\mathbb{R}$ be a potential whose rotation set has the form $[r,0]$ for some $r\in(-\infty,0)$, such that there is an invariant measure $\mu_0$ realizing $\mu_0(\pot) = 0$ and maximizing entropy among such measures: $h(T,\mu_0)=\max \{h(T,\mu) \colon \mu\in \Prob(T), \mu(\pot)=0\}$.
The following are equivalent:
\begin{enumerate}
\item\label{enumi:fpt1} the linear thermodynamical formalism for the system $(T,\pot)$ exhibits a freezing phase transition, i.e., for some $\beta_0>0$ and all $\beta>\beta_0$, the set of equilibrium measures is non-empty and independent of $\beta$,
\item\label{enumi:fpt2} there is some finite $\beta$ such that $\mu_0$ is an equilibrium measure for $\beta\pot$,
\item\label{enumi:fpt3} the topological pressure function
\begin{align*}
\Pf \colon \mathbb{R} &\to \mathbb{R} \\
  \beta &\mapsto \sup \{ h(T,\mu) + \beta\mu(\pot) \colon \mu\in\Prob(T) \}
\end{align*}
is affine on some interval $[\beta_0,+\infty)$,
\item\label{enumi:fpt4} there exists $C>0$ such that $h(T,\mu) \le h(T,\mu_0) -C \mu(\pot)$ for all $\mu\in\Prob(T)$.
\end{enumerate}  
When these conditions are realized, the critical inverse temperature, i.e., the least possible value of $\beta_0$, is the least possible $C$ in the entropy-potential inequality (\ref{enumi:fpt4}). The intercept of the affine part of the graph of $\Pf$ is then the entropy of equilibrium measures after the freezing phase transition, and its slope is their energy $\mu(\pot)$ (here $0$  is given  by the chosen normalization of the rotation set).
\end{proposition}

\begin{proof}
The main novelty here is the observation that (\ref{enumi:fpt4}) characterizes Freezing Phase Transitions, but for the sake of completeness we prove all the equivalences, through the cycle $\text{\it(\ref{enumi:fpt1})}\implies \text{\it(\ref{enumi:fpt3})} \implies \text{\it(\ref{enumi:fpt4})} \implies \text{\it(\ref{enumi:fpt2})} \implies \text{\it(\ref{enumi:fpt1})}$.

Assume \textit{(\ref{enumi:fpt1})} and let $\mu_1$ be any equilibrium measure for any $\beta>\beta_0$. For all $\beta>\beta_0$ we get $\Pf(\beta) = h(T,\mu_1)+\beta\mu_1(\pot)$, an affine expression.

Convex duality translates angular points to flat regions and vice-versa; that $\Pf$ is affine on an interval means that the entropy-potential diagram has an angular point with a supporting line of slope $-\beta$ for each $\beta$ in the interval. 
 Let us explain this, a simple case of what we left hidden behind the appeal to Legendre duality above. Using the notation $\hf(z)=\sup\{h(T,\mu) \colon \mu(\pot)=z\}$ for all $z\in[r,0]$, $\hf$ is concave thus continuous on $(r,0)$, and has a continuous extension $\bar\hf$ on $[r,0]$. We can the rewrite $\Pf(\beta)=\max_z \bar\hf(z)+\beta z$. Denoting by $z_\beta$ an abscissa realizing $\Pf(\beta)$, observe that for all $\eps>0$, $\Pf(\beta+\eps)\ge \bar\hf(z_\beta)+(\beta+\eps) z_\beta\ge \Pf(\beta)+\eps z_\beta$ so that the right derivative of $\Pf$ is at least $z_\beta$. Similarly, $\Pf(\beta-\eps)\ge \Pf(\beta)-\eps z_\beta$ shows that the left derivative is at most $z_\beta$. Whenever $\Pf$ is differentiable, $\Pf'(\beta)=z_\beta$. On an affine part, the derivative exists and is constant, therefore $z_\beta$ is (locally) constant and $\hf$ has an angular point.
Moreover the abscissa of the angular point is the slope of the line extending the affine part of the graph of $\Pf$, while the ordinate of that point is the intercept of that line. 

Item \textit{(\ref{enumi:fpt3})} thus implies that the entropy-potential diagram has an angular point with supporting lines of slope $-\beta$ for all $\beta\ge\beta_0$. Since slopes are arbitrarily high in magnitude, the abscissa of this angular point must be the supremum of the rotation set, i.e., $0$. It must then have ordinate equal to the supremum of the realizable entropies for this energy, i.e., $h(T,\mu_0)$. In particular, the entropy-potential diagram is constrained under a line of equation $(h(T,\mu)=h(T,\mu_0)-\beta_0 \mu(\pot))$, which is \textit{(\ref{enumi:fpt4})}.

Assume \textit{(\ref{enumi:fpt4})} and take any $\beta\ge C$. For all $\mu\in\Prob(T)$,
\[h(T,\mu)+\beta\mu(\pot) \le h(T,\mu_0) + (\beta-C) \mu(\pot) \le h(T,\mu_0) = h(T,\mu_0)+\beta\mu_0(\pot)\]
so that $\mu_0$ is an equilibrium measure for such $\beta$, proving \textit{(\ref{enumi:fpt2})}.

Assume \textit{(\ref{enumi:fpt2})}, let $\beta_1$ be such that  $\mu_0$ is an equilibrium measure for $\beta_1\pot$ and $\beta>\beta_1$. For all $\mu\in\Prob(T)$, since $\mu(\pot)\le 0$ and $\mu_0(\pot)=0$,
\[h(T,\mu)+\beta\mu(\pot) \le h(T,\mu)+\beta_1\mu(\pot) \le h(T,\mu_0)+\beta_1\mu_0(\pot) = h(T,\mu_0)+\beta\mu_0(\pot) \]
and $\mu_0$ is an equilibrium measure for $\beta\pot$. It follows that the set of $\beta$'s such $\mu_0$ is an equilibrium measure for $\beta\pot$ is an interval $\beta_0,+\infty)$. The above computation shows that for all $\beta>\beta_0$, the set of equilibrium measure is $\{\mu\in\Prob(T) \colon \mu(\pot)=0, h(T,\mu)=h(T,\mu_0) \}$, and is thus independent of $\beta$.
\end{proof}

\begin{remark}
If we consider several potentials $\pot_1,\dots,\pot_d$, the condition in Legendre regularity that $\lvert\grad \hf\rvert$ goes to $+\infty$ as one approaches the boundary is violated exactly when some linear combination of the $(\pot_k)$ exhibit a (linear) freezing phase transition.
\end{remark}

The entropy-potential diagram makes it {\JB clear} how to prove existence of freezing phase transition in both the linear and nonlinear settings. We divide Theorem \ref{thm-fpt} of the introduction in two parts.

\begin{theorem}\label{thm-fpt-linear}
Let $T:X\to X$ be a continuous {\BK map} of finite, positive topological entropy such that $\mu\mapsto h(T,\mu)$ is upper semi-continuous. Consider  $\mu_0\in\Prob_\erg(T)$ with zero entropy. Then there exists a continuous potential $\pot:X\to \mathbb{R}$ such that the linear thermodynamical formalism of $(T,\pot)$ exhibits a freezing phase transition with ground state $\mu_0$. Moreover we can ensure that $\mu_0$ is the unique ground state, and that at the critical inverse temperature $\beta_0$ there are exactly two equilibrium states.
\end{theorem}

In particular, if $K$ is a compact $T$-invariant {\JB set} with zero topological entropy, then we can find a potential exhibiting a freezing phase transition supported on $K$. This broadly extends \cite{Bruin-Leplaideur1, Bruin-Leplaideur2} by proving existence of freezing phase transitions for all zero-entropy subshifts, instead of very specific ones; but it is not constructive, since the potential $\pot$ is ultimately obtained through the Hahn-Banach theorem.

\begin{proof}
According to a Theorem of Jenkinson \cite{Jenkinson}, there exists a continuous potential $\tilde\pot:X\to\mathbb{R}$ such that $\mu_0$ is the unique equilibrium state of $\tilde\pot$, i.e., the unique maximizer of $h(T,\mu)+\beta\mu(\tilde\pot)$ for $\beta=1$. Since $h(T,\mu_0)=0$ is minimal, $\mu_0$ must be a maximizing measure for $\tilde\pot$. The conclusion then follows from Proposition \ref{prop-fpt} applied to the adjusted potential $\pot=\tilde\pot-\mu_0(\tilde\pot)$.

To have a second equilibrium state at the critical inverse temperature, it suffices to consider an arbitrary ergodic measure $\mu_1$ of positive entropy: Jenkinson's theorem provides a continuous potential whose only ergodic equilibrium states (at $\beta=1$) are $\mu_0$ and $\mu_1$. This also fixes the critical inverse temperature at $\beta_0=1$.
\end{proof}

\begin{theorem}\label{thm-fpt-nonlinear}
Let $T:X\to X$ be a continuous dynamical system of finite, positive topological entropy such that $\mu\mapsto h(T,\mu)$ is upper semi-continuous. Let $\pot:X\to(-\infty,0]$ be a continuous potential such that $K=\pot^{-1}(0)$ is $T$-invariant and has zero topological entropy. 

Then there exists a continuous nonlinearity $F_1 : (-\infty,0]\to (-\infty,0]$ with $F(0)=0$ such that the energy $\En_1(\mu) = F_1(\mu(\pot))$ exhibits a ``strong freezing phase transition'' in the following sense. There is a $\beta_0>0$ such that:
\begin{itemize}
\item for each $\beta<\beta_0$ the energy $\beta \En_1$ has at least one equilibrium measure, and none of them are supported on $K$,
\item at $\beta=\beta_0$ there are several equilibrium measures, at least one supported on $K$ and one not supported on $K$,
\item for each $\beta>\beta_0$ the equilibrium measures are exactly the $K$-supported, $T$-invariant measures and the topological pressure function $\beta\mapsto \NLP^{\beta\En_1}_\top(T)$ is affine.
\end{itemize}
\end{theorem}

Observe that here $F_1$ will only be continuous at $0$; we can extend it continuously to $\mathbb{R}$, but we cannot make $F_1$ differentiable in a neighborhood of $0$.

\begin{proof}
Take for $F_1$ any increasing convex continuous function  $(-\infty,0]\to(-\infty,0]$ such that $\hf(z) = o(-F_1(z))$ as $z\to 0$. Theorem \ref{thm-var-prin-NL} ensures that equilibrium measures are found by optimizing $\hf(z)+\beta F_1(z)$ and then maximizing entropy in $\mathcal{M}(z)$, as in Section \ref{sec:convexity}
(we did not assume Legendre regularity, but we assumed enough to ensure that each optimal $z$ comes with at least one equilibrium measure).

Since $\hf$ is bounded by $h_\top(T)$, for $\beta$ large enough the graph of $-\beta F_1$ is above the graph of $\hf$ except at $0$ where they meet. 
This means that for these $\be$s, $\hf(z)+\be F_1(z)$ is non positive and always negative for $z<0$, i.e., the unique optimal $z$ is $0$.

Let $\beta_0$ the least $\beta$ such that $\hf(z)\le -\beta F_1(z)$ for all $z$. Since $\hf(z) = o(-\beta_0 F_1(z))$ as $z\to 0$, there must be a touching point distinct from $0$, and we get two optimal values $z=0$ and $z=z_0<0$, and at least two equilibrium measures. For $\beta<\beta_0$, $z=0$ cannot be optimal anymore since $z_0$ is strictly better. The conclusion follows.
\end{proof}

A simple example can be worked out in the case of the shift over $X=\{a,b\}^{\mathbb{N}}$ and the potential $\pot$ taking the values $0$ on the cylinder $[a]$ and $-1$ on the cylinder $[b]$. We have $\hf(z) \sim z\log(-z)$ at zero, so that we can take $F_1(z) = -(-z)^\alpha$ with any $\alpha\in(0,1)$: the nonlinear thermodynamical formalism associated with the energy
 $$
    \mu\mapsto -\lvert \mu(\pot)\rvert^\alpha
 $$
exhibits a strong freezing phase transition with ground state $\mu_0=\delta_{aaa\dots}$.

\bigbreak

\end{document}